\documentclass[reqno]{amsproc}
\usepackage[utf8]{inputenc}
\usepackage{amsfonts}
\usepackage{orcidlink}
\usepackage{graphicx}
\usepackage{amscd}
\usepackage{amsmath}
\usepackage{amssymb}
\usepackage{latexsym}
\usepackage{hyperref}
\usepackage[all]{xy}
\usepackage{color}
\usepackage{mathrsfs}
\theoremstyle{plain}

\newtheorem{example}{\bf Example}
\newtheorem{lemma}{\bf Lemma}

\newtheorem{proposition}{\bf Proposition}
\newtheorem{remark}{\bf Remark}

\newtheorem{theorem}{\bf Theorem}
\newtheorem{theorem*}{Theorem}{}
\numberwithin{equation}{section}
\newcommand\dv{\mathrm{div}}
\newcommand\dm{\mathrm{dm}}
\newcommand\dM{\mathrm{dM}}
\newcommand\dn{\mathrm{d\mu}}

\newcommand\sym{\mathrm{Sym}}
\newcommand\dmt{\mathrm{dm_t}}
\newcommand\dnt{\mathrm{d\mu_t}}
\newcommand\dg{\mathrm{dm_0}}
\newcommand\db{\mathrm{d\mu_0}}
\newcommand\ricci{\mathrm{Ric}}
\newcommand\drift{\mathrm{L}}
\begin{document}
\title[On eigenvalues of a class of second order elliptic operators]{Hadamard-type variation formulae for the eigenvalues of a class of second-order elliptic operators and its applications}

\author{C.L. Cunha$^1$, J.N.V. Gomes$^2$\orcidlink{0000-0001-5678-4789}}
\author{M.A.M. Marrocos$^{3,*}$\orcidlink{0000-0001-6336-0693}}

\address{$^1$Instituto de Educação, Agricultura e Ambiente, Universidade Federal do Amazonas, Rua 29 de Agosto, 786, 69.800-000 Humaitá, Amazonas, Brazil.}
\address{$^2$Departamento de Matemática, Universidade Federal de São Carlos, Rod. Washington Luís, Km 235, 13.565-905 São Carlos, São Paulo, Brazil.}
\address{$^3$Departamento de Matem\'atica, Universidade Federal do Amazonas, Av. General Rodrigo Oct\'avio, 6200, 69.080-900 Manaus, Amazonas, Brazil}

\email{$^1$ncleiton@ufam.edu.br} 
\email{$^2$jnvgomes@ufscar.br}
\email{$^3$marcusmarrocos@ufam.edu.br (*corresponding author)}

\urladdr{$^{1,3}$https://www.ufam.edu.br}
\urladdr{$^2$https://www.ufscar.br}



\begin{abstract}
We use variational methods to derive Hadamard-type formulae for the eigenvalues of a class of elliptic operators on a compact Riemannian manifold $M$. We then apply the latter in the following context. Consider a family of elliptic operators which is parametrized by either the set of all $\mathcal{C}^r$--Riemannian metrics on $M$ or the set of all $\mathcal{C}^r$--diffeomorphisms on a domain into $M$. In either case, we prove that if a subset of the parametrizations set yields a simple spectrum of the operator, then it is necessarily a generic subset. We also analyse the behavior of the eigenvalues when the metric evolves along the Ricci flow on a closed Riemannian manifold, and we prove, under a suitable hypothesis, that they increase.
\end{abstract}

\maketitle

\section{Introduction}
Let $(M^n,g)$ be an $n$--dimensional compact Riemannian manifold with boundary $\partial M$, and let $\mathscr{L}$ be the class of elliptic differential operators given by
\begin{equation}\label{defL}
\mathscr{L}f:=\dv_\eta(T\nabla f):=\dv(T\nabla f)-\langle \nabla\eta,T\nabla f\rangle,
\end{equation}
where $T$ is a symmetric positive definite $(1,1)$--tensor on $M^n$ and $f,\eta\in C^{\infty}(M)$. Here, $\dv$ stands for the divergence of smooth vector fields and $\nabla$ for the gradient of smooth functions. We allow both $T$ and $\eta$ depend on the metric $g=\langle,\rangle$. This class of operators unifies several particular cases that are well-studied in the literature, e.g., the Laplace-Beltrami operator, the drifted Laplacian, the Cheng-Yau operator, among other geometric operators which naturally arise in geometric analysis.
For instance, when $T$ is divergence free, then the operator $\mathscr{L}$ is likely to have applications in physics, see, e.g., Serre~\cite{Serre}. We highlight that Serre's work deals with divergence free positive definite symmetric tensors and fluid dynamics, there we can find examples and know where these tensors occur. We refer the reader to Alencar, Neto and Zhou~\cite{AGD}, Araújo Filho and Gomes~\cite{MCAFandGomes}, Gomes and Miranda~\cite{miranda} or Navarro~\cite{Navarro} for more related discussions.

In this paper, we study the behavior of the spectrum of $\mathscr{L}$ by means of Hadamard-type variational formulae for its eigenvalues. It is organized as follows. We give some preliminaries in Section~\ref{prem} and some general comments on variational formulae in Section~\ref{var}. In Section~\ref{Section Dbc}, we consider the eigenvalue problem for the operator $\mathscr{L}_{g}$ with the Dirichlet boundary condition. Theorem~\ref{thm1} deal with the generic simplicity of the eigenvalues of the operator $\mathscr{L}_g(\cdot)= \dv_\eta (T_g \nabla\cdot)$, where the $T_g$--family satisfies the property $\mathcal{P}$. A precise description of this latter property is given in Section~\ref{Section GenDirc}. Taking into account property $\mathcal{P}$, we prove that all eigenvalues of $\mathscr{L}_g$ are generically simple. The perturbation of the domain case is settled by Proposition~\ref{pro domain} in which the $\eta$--dependency on the parameter $t$ is crucial. In Section~\ref{Section Nbc}, we deal with the generic simplicity of the eigenvalues of the operator $\mathscr{L}_g$ with the Neumann boundary condition. In Section~\ref{DV}, we deal with domain variation considering both boundary conditions, Dirichlet and Neumann, and proving that the eigenvalues are generically simple. Especially, Theorem~\ref{thm2} shows that the multiplicity of an eigenvalue can be reduced by small perturbation of the domain. In Section~\ref{ApExtremal}, we deal with extremal metrics. A domain $\Omega$ is local minimizer (local maximizer) for the $k^{th}$--eigenvalue $\mu_k$ of $\mathscr{L}$ if for any analytic volume-preserving deformation $\Omega_t$, the function $t\mapsto \mu_k(t)$ admits a local minimum (local maximum) at $t=0$. For the Dirichlet Laplacian case, it is known that if a domain is local minimizer (resp. local maximizer) for $\mu_k$ with $\mu_k>\mu_{k-1}$ (resp. $\mu_k<\mu_{k+1}$), then the norm $|\partial\phi/\partial\nu|$ of the normal derivative of the eigenfunction $\phi$ of $\mu_k$ is constant (see~\cite{soufi}). We prove a similar result for $\mathscr{L}$ in Theorem \ref{thm3}. In Section~\ref{ApRF}, we analyze the behavior of the eigenvalues of the operator $\mathscr{L}_{g(t)}$ along the Ricci Flow. For instance, we prove a monotonicity for one-parameter family of eigenvalues $\lambda(t)$ of $\mathscr{L}_{g(t)}$ along the Ricci flow on a closed homogeneous $n$--dimensional Riemannian manifold, see Theorem~\ref{rem lm Ricci}. We also prove the monotonicity of the eigenvalues of $\dv_{\eta}(\psi\nabla u)$ envolving by the Ricci flow on three--dimensional Riemannian manifold having strictly positive Ricci curvature initially, see Theorem~\ref{thmMono}. In particular, since the solution to the Ricci flow becomes extinct in finite time, with additional hypothesis over the function $\psi$, we prove that $\lim\limits_{t\rightarrow \delta}\lambda(t)=\infty$, see Theorem~\ref{AAAE}.

It is worth mentioning here some of the main articles related to our work. Uhlenbeck~\cite{uhlenbeck} showed that on a closed Riemannian manifold $(M^n,g)$ there exists a residual subset $\Gamma$ into $\mathcal{M}^r$ for which the spectrum of the Laplace-Beltrami operator $\Delta_g$ is simple, where $\mathcal{M}^r$, $2\leq r <\infty$, stands for the set of all $\mathcal{C}^r$--Riemannian metrics $g$ on $M^n$.

Canzani~\cite{canzani} considered elliptic, formally self-adjoint, conformally covariant operator $P_g$ of order $m$ acting on smooth sections over a closed Riemannian manifold $(M^n,g)$, and proved, under special conditions over the eigenspaces of $P_g$, that the set of all functions $f\in C^\infty(M)$ for which $P_{e^fg}$ has only simple nonzero eigenvalues is a residual subset of $C^\infty(M)$.

El Soufi and Ilias~\cite{soufi} established necessary and sufficient conditions for a domain to be critical, locally minimizing or locally maximizing the $k^{th}$--eigenvalue of the Laplace-Beltrami operator. As an application, they obtained a characterization for critical domains of the trace of the heat kernel under Dirichlet boundary condition.

Cao, Hou and Ling~\cite{Cao_annalen} showed a monotonicity Hadamard-type formula for the first eigenvalue of $-\Delta + aR$ $(0 < a < 1/2)$ on a closed surface with nonnegative scalar curvature $R$ under the Ricci flow. Recall that this latter flow was introduced by Hamilton~\cite{hamilton1} to study the geometry of positive Ricci curvature on three--dimensional closed Riemannian manifolds. Currently, the Ricci flow stands as a powerful tool in studying the geometry and topology for lower dimensional manifolds. 

\section{Preliminaries}\label{prem}

Let $\dm=e^{-\eta}\dM$ be the weighted volume form on $(M^n,g)$, $\dn=e^{-\eta}\mathrm{d\sigma}$ be the weighted area form induced on $\partial M$, where $\dM$ is the Riemannian volume form of $g=\langle,\rangle$, and $\eta:M\to\mathbb{R}$ be a smooth function. If $T$ is a $(0,2)$--tensor on $M^n$, then we can associate it with a unique $(1,1)$--tensor, which will also be denoted $T$, according to $\langle TX,Y\rangle= T(X,Y)$ for all $X,Y\in\mathfrak{X}(M)$. Thus, for an $(1,1)$--tensor $S$, we have
\begin{equation*}
ST(X,Y)= \langle STX,Y\rangle =S(TX,Y).
\end{equation*}

Recall that the divergence of an $(1,1)$--tensor $T$ is the $(0,1)$--tensor
\begin{equation*}
(\dv T)(v)(p) = \mathrm{tr}_g\big(w \mapsto (\nabla_w T)(v)(p)\big),
\end{equation*}
for $p\in M^n$ and $v,w\in T_pM.$ Recall also that the inner product induced by $g$ on the space of $(0,2)$--tensors on $M^n$ is $\langle T,S\rangle=
\mathrm{tr}_g\big(TS^{*}\big)$, where $S^*$ denotes the adjoint tensor of $S$. Clearly, we have in a local coordinates
\begin{equation*}
\langle T,S\rangle =g^{ik}g^{jl}T_{ij}S_{kl}.
\end{equation*}
In this way, if $\nabla^2f=\nabla df$ stands for the Hessian of $f$, then $\Delta_gf=\langle\nabla^2f,g\rangle$. Now, we define
the $(\eta,T_g)$--divergence operator
\begin{equation*}
\mathscr{L}_{g}f=\dv(T_g\nabla f)-T_g(\nabla\eta,\nabla f) = e^{\eta} \dv(e^{-\eta}T\nabla f),
\end{equation*}
where $T_g$ is a symmetric positive definite $(0,2)$--tensor which depends smoothly on $g$.

If we consider the drifted divergence operator $\dv_{\eta}X:=\dv X-\langle\nabla\eta,X\rangle$, then the $(\eta,T_g)$--divergence operator takes the form $\mathscr{L}_gf=\dv_{\eta}(T\nabla f)$, moreover, applying the divergence theorem, one has 
\begin{align*}
&\int_M\dv_{\eta}(TX)\dm=\int_{\partial M}\langle TX,\nu\rangle\dn,
\end{align*}
where $\nu$ is the outward pointing unit normal vector field along $\partial M$. 

Since $\dv_{\eta}(T(fX))=f\dv_\eta(TX)+\langle \nabla f,TX\rangle$  for  $f,\ell\in C^\infty(M)$, we obtain
\begin{align}\label{Formula-IP}
&\int_{M}\ell \dv_{\eta}(T\nabla f)\dm=-\int_{M}T_g(\nabla\ell,\nabla f)\dm+\int_{\partial M}\ell T_g(\nabla f,\nu)\dn
\end{align}
and the following formulae hold
\begin{align}
&\mathscr{L}_g(f\ell) = f\mathscr{L}_g\ell + \ell\mathscr{L}_gf + 2 T_g(\nabla f, \nabla\ell).\nonumber
\end{align}
 
It follows from~\eqref{Formula-IP} that $(\eta,T_g)$--divergence operator is a formally self-adjoint operator in the Hilbert space of all real functions $f\in L^2(M,\dm)$ that vanish (or $T_g(\nabla f,\nu)=0$) on $\partial M$ in the sense of the trace. Thus, the Dirichlet (or $T_g$--Neumann) boundary eigenvalue problem for $\mathscr{L}_{g}$ has real and discrete spectrum 
\begin{equation*}
0 = \mu_{0}(g) < \mu_{1}(g)\leq\mu_{2}(g)\leq \cdots  \leq\mu_{k}(g)\leq\cdots\to\infty,
\end{equation*}
where each $\mu_i(g)$ is repeated according to its multiplicity. Moreover, since $M^n$ is compact there exist two positive constants $\alpha$ and $\beta$, such that $$0<\alpha |Y|^2\leq T_g(Y,Y)\leq \beta |Y|^2,$$ for all $Y \in \mathfrak{X}(M)$. Now, as $\mathscr{L}_{g}$ is a closed operator, we can use perturbation theory for linear operators as in Kato~\cite{kato}.

Let $\mathcal{S}^r(M)$ be the Banach spaces of $C^r$--symmetric $(0,2)$--tensors on $M^n$, and let $\mathcal{M}^r$ be the Banach spaces of $C^r$--Riemannian metrics on $M^n$, $2\leq r<\infty$. A subset of $\mathcal{M}^r$ is called {\it{residual}} if it contains a countable intersection of open dense sets. A property of metrics in $\mathcal{M}^r$ is called {\it{generic}} if it holds on a residual subset.

\section{Variation formulae by means of the metrics}\label{var}

We start by considering $t\mapsto g(t)$ a smooth variation of an initial metric $g$. Let $\dmt=e^{-\eta}\dM_t$ be the weighted volume form of $g(t)$, and $\dnt=e^{-\eta}\mathrm{d\sigma}_t$ be the weighted area form induced on $\partial M$. Denoting by $H$ the $(0,2)$--tensor given by $H_{ij}=\frac{d}{dt}\big|_{t=0}g_{ij}(t)$ and writing $h=g^{ij}H_{ij}$, we get $\frac{d}{dt}\dm_t=\frac{1}{2}h\dm$. Similarly, if $\tilde{H}$ stands for the $(0,2)$--tensor induced by the derivative of $g(t)$ restricted to $\partial M$, and 
$\tilde{h}$ denotes its trace, then $\frac{d}{dt}\dnt = \frac{1}{2}\tilde{h}\dn$. Since there is no danger of confusion, we also write $\langle,\rangle$ for the metric $g(t)$. Thus, for each $X\in \mathfrak{X}(M)$, we can write $X_t=g^{ij}(t)x_i(t)\partial_j$, where $x_i(t)=\langle X,\partial_i\rangle$, and we define $\dot{X}:=g^{ij}x_i'\partial_j$, where $x_i'=\frac{d}{dt}x_i(t)$. 

Now, we consider $t\mapsto T_{g(t)}=:T_t$ to be a smooth variation of the initial symmetric $(0,2)$--tensor $T=T_g$. Then, we obtain the identity
\begin{equation}\label{d/dt T}
\frac{d}{dt}T_t(X_t,Y_t) =T(\dot{X},Y)+T(X,\dot{Y})+\mathscr{H}_T(X,Y),
\end{equation}
where
\begin{equation*}
\mathscr{H}_T:=-(TH+HT)+T'\quad and \quad T'(\partial_i,\partial_j)=\frac{d}{dt}T_t(\partial_i,\partial_j)=(T_t)'_{ij}.
\end{equation*}
Indeed, a straightforward computation gives us 
\begin{equation*}
T_t(X_t,Y_t)=g^{ij}(t)g^{kl}(t)x_i(t)y_k(t)(T_t)_{jl} \quad \mbox{and} \quad  \frac{d}{dt}g^{ij}(t)=-g^{ir}(t)g^{js}(t)H_{rs}(t).
\end{equation*}
Hence
\begin{align*}
\frac{d}{dt}T_t(X_t,Y_t) =&\quad g^{ij}g^{kl}x_i'(t)y_k(t)(T_t)_{jl}+g^{ij}g^{kl}x_i(t)y_k'(t)(T_t)_{jl}\\
&-g^{ir}g^{js}g^{kl}H_{rs}x_iy_k(T_t)_{jl}-g^{ij}g^{kr}g^{ls}H_{rs}x_iy_k(T_t)_{jl} + g^{ij}g^{kl}x_iy_k(T_t)_{jl}'
\end{align*}
which is enough to obtain~\eqref{d/dt T}. 

In the special case of gradient vector fields, we denote by $\nabla_t f$ the gradient of $f\in C^\infty(M)$ computed in the metric $g(t)$. Observe that $\nabla_t f=g^{ij}(t)f_i\partial_j$, where $f_i=\langle \nabla f,\partial_i\rangle=\partial_if$ does not depend on $t$. Thus, from~\eqref{d/dt T} we have 
\begin{equation}\label{eq1lem1}
\frac{d}{dt}T_t(\nabla_t f,\nabla_t\ell)=\mathscr{H}_T(\nabla f,\nabla\ell)
\end{equation}
for all $f,\ell\in C^{\infty}(M)$. In particular,
\begin{equation*}
\frac{d}{dt}\langle\nabla_t f,\nabla_t\ell\rangle=-H(\nabla f,\nabla\ell).
\end{equation*}
Moreover, if we consider the smooth curve $t\mapsto \ell(t)\in C^\infty(M)$ and the smooth vector field $\nu_t:=\frac{\nabla_t f}{|\nabla_t f|}$, then
\begin{equation}\label{Neumann-eq3lem2}
\frac{d}{dt}T_t(\nu_t,\nabla_t \ell(t))=\frac{1}{2}H(\nu,\nu)T(\nu,\nabla \ell)+T(\nu,\nabla \ell')+\mathscr{H}_T(\nu,\nabla \ell).
\end{equation}
Indeed, first note that $\nu_i=\frac{1}{|\nabla f|}\langle\nabla f,\partial_i\rangle$ implies
\begin{equation}
\nu_i'=\frac{1}{2}\frac{1}{|\nabla f|}H(\nu,\nu)\partial_if,
\end{equation}
and then
\begin{equation*}
T(\dot{\nu},\nabla \ell)=\frac{1}{2}H(\nu,\nu)T(\nu,\nabla \ell).
\end{equation*}
Thus, by taking $X_t=\nu_t$ and $Y_t=\nabla_t\ell(t)$ into \eqref{d/dt T}, we obtain \eqref{Neumann-eq3lem2}.

Now, we consider a smooth function $\eta:I\times M\rightarrow\mathbb{R}$ and write $\dot{\eta}:=\frac{d}{dt}\big|_{t=0}\eta(t)$ to define
\begin{equation}\label{def.barL}
\bar{\mathscr{L}}_{t}\phi:=\dv(T_t\nabla \phi) - T_t(\nabla\eta(t),\nabla\phi)
\end{equation}
for all $\phi\in C^{\infty}(M)$. We observe that the family of operators $\bar{\mathscr{L}}_t$ is very convenient to deal with domain deformations.

\begin{lemma}\label{lemEq.bar.L'}
For all $f\in C_{c}^{\infty}(M)$,

\begin{equation}\label{eq.L'}
\bar{\mathscr{L}}'f = \dv_{\eta}(\mathscr{H}_T\nabla f)+\frac{1}{2}T(\nabla h,\nabla f)-T(\nabla\dot{\eta},\nabla f)
\end{equation}
where $\bar{\mathscr{L}}':=\frac{d}{dt}\big|_{t=0}\bar{\mathscr{L}}_{t}$.
\end{lemma}

\begin{proof}
We begin by considering $\eta$ to be independent of $t$. In this case, $\bar{\mathscr{L}}_t$ becomes $\mathscr{L}_t$. Integration by parts formula gives us
\begin{equation*}
\int_{M}\ell \mathscr{L}_{t}f\dm_t=-\int_{M}T_t(\nabla\ell,\nabla f) \dm_t
\end{equation*}
for $\ell\in C_{c}^{\infty}(M)$. Hence, from equation \eqref{eq1lem1}, we have at $t=0$
\begin{equation}\label{eq2lem1}
\int_M\ell \mathscr{L}'f\dm = \int_M\Big(-\frac{1}{2}\ell h\mathscr{L}f -\mathscr{H}_T(\nabla f,\nabla \ell)-\frac{1}{2}hT(\nabla\ell,\nabla f) \Big)\dm.
\end{equation}
Observe that
\begin{equation}\label{eq3lem1}
\dv_{\eta}(\mathscr{H}_T(\ell\nabla f))=\ell\dv_{\eta}(\mathscr{H}_T\nabla f)+\mathscr{H}_T(\nabla f,\nabla \ell)
\end{equation}
and
\begin{equation}\label{eq4lem1}
\dv_{\eta}(\ell hT\nabla f)=\ell h\mathscr{L}f+\ell T(\nabla h,\nabla f)+hT(\nabla\ell,\nabla f).
\end{equation}
Plugging \eqref{eq3lem1} and \eqref{eq4lem1} in \eqref{eq2lem1}, we obtain
\begin{equation*}
\int_{M}\ell \mathscr{L}'f \dm=\int_M \ell\left(\frac{1}{2}T(\nabla h,\nabla f) +\dv_{\eta}(\mathscr{H}_T\nabla f)\right)\dm,
\end{equation*}
which is sufficient to prove~\eqref{eq.L'} for this particular case. For the general case, it is enough to note that
\begin{equation}\label{def.barL0}
\bar{\mathscr{L}}'f = \mathscr{L}'f-T(\nabla\dot{\eta},\nabla f)
\end{equation}
to conclude the proof of the lemma.
\end{proof}

\section{Hadamard-type variation formula}\label{Section Dbc}

In what follows, we assume that all manifolds are oriented and those that are compact we assume to have boundary. The closed manifolds are assumed to be compact without boundary.

As in the previous section, we continue by considering $t\mapsto T_{g(t)}=:T_t$ a smooth variation of the initial symmetric $(0,2)$--tensor $T=T_g$, and the family of operators $\bar{\mathscr{L}}_t$ as in \eqref{def.barL}, besides, we denote $d\bar m_t =e^{-\eta(t)}dM_t$.

The next result is an extension  of Lemma $3.15$ in Berger \cite{berger} for the $(\eta,T_g)$--divergence operator $\mathscr{L}_{g}$.
\begin{lemma}\label{LemExist}
Let $(M,g)$ be a compact Riemannian manifold. Consider a real analytic one-parameter family of Riemannian metrics $g(t)$ on $M$ with $g=g(0)$. If $\lambda$ is an eigenvalue of multiplicity $m>1$ for $\mathscr{L}_{g}$, then
there exist $\varepsilon>0$, scalars $\lambda_{i}$ ($i=1,\dots,m$) and functions $\phi_{i}$ varying analytically in $t$, such that, for
all $|t|<\varepsilon$ the following holds:
\begin{enumerate}
\item $\lambda_{i}(0)=\lambda$;
\item $\{\phi_{i}(t)\}$ is orthonormal in $L^2(M,d\bar m_t)$;
\item $\left\{\begin{array}{ccccc}
               -\bar{\mathscr{L}}_{t}\phi_{i}(t)&=&\lambda_{i}(t)\phi_{i}(t)& \mbox{in} & M\\
               \phi_i(t)&=&0 &\mbox{on} &\partial M.
               \end{array}\right.$
\end{enumerate}
\end{lemma}

\begin{proof}
Write $g(z)$ and $T_z$ for the extensions of $g(t)$ and $T_t$ to a domain $D_0$ of the complex plane $\mathbb{C}$, respectively. So, we can naturally extend the operator $\bar{\mathscr{L}}_t$ to
\begin{equation*}
\bar{\mathscr{L}}_{z}:=\bar{\mathscr{L}}_{a+ib}=\bar{\mathscr{L}}_{a}+i\bar{\mathscr{L}}_{b}.
\end{equation*}

We observe that the domain $D=H^2(M)\cap H_0^1(M)$ of $\bar{\mathscr{L}}_{z}$ do not depend on $z$, and since $M$ is compact, for any two Riemannian metrics $g_1,g_2$ the associated Sobolev norms $||.||_{W^{1,2}(M,g_i)}$ in $H^2(M)$ are equivalent. Moreover, the map $z\mapsto \bar{\mathscr{L}}_{z}\phi$ is holomorphic for $z\in D_0$ and for every $\phi\in D$. Thus, $\bar{\mathscr{L}}_{z}$ is a holomorphic family of type $(A)$ as in Kato~\cite{kato}. Once the family $\bar{\mathscr{L}}_{z}$ is not self-adjoint in its domain, we can consider the isometry $P_t:L^2(M,\dm)\rightarrow L^2(M,d\bar{m}_t)$ given by
\begin{equation*}
P_t(u)=\frac{e^{-\eta(0)/2}\sqrt[4]{det(g_{ij}(0))}}{e^{-\eta(t)/2}\sqrt[4]{det(g_{ij}(t))}}u,
\end{equation*}
so that each operator $\tilde{\mathscr{L}_{t}}:=P_t^{-1}\circ \bar{\mathscr{L}}_{t}\circ P_t$ is self-adjoint in $L^2(M,\dm)$. Note that $P^{-1}_t(u)=\frac{e^{-\eta(t)/2}\sqrt[4]{det(g_{ij}(t))}}{e^{-\eta(0)/2}\sqrt[4]{det(g_{ij}(0))}}u$ and 
$d\bar{m}_t=\frac{e^{-\eta(t)}\sqrt[2]{det(g_{ij}(t))}}{e^{-\eta(0)}\sqrt[2]{det(g_{ij}(0))}} \dm$. Hence,
\begin{align*}
\int_{M}v\tilde{\mathscr{L}_{t}}u\dm =& \int_{M}vP_t^{-1}\circ \bar{\mathscr{L}}_{t}\circ P_t(u)\dm= \int_{M}(P_t^{-1}\circ P_t(v))(P_t^{-1}\circ \bar{\mathscr{L}}_{t}\circ P_t(u))\dm \\
=& \int_{M}(P_t(v))( \bar{\mathscr{L}}_{t}\circ P_t(u)) \frac{e^{-\eta(t)}\sqrt[2]{det(g_{ij}(t))}}{e^{-\eta(0)}\sqrt[2]{det(g_{ij}(0))}}\dm\\
=&\int_{M}P_t(v)\bar{\mathscr{L}}_{t}P_t(u)d\bar{m}_{t}=\int_{M}P_t(u)\bar{\mathscr{L}}_{t}P_t(v)d\bar{m}_{t}\\
=&\int_{M}(P_t^{-1}\circ P_t(u))(P_t^{-1}\circ\bar{\mathscr{L}}_{t}\circ P_t(v))\dm=\int_{M}u\tilde{\mathscr{L}_{t}}v\dm.
\end{align*}
Evidently, $\tilde{\mathscr{L}_{t}}$ share the same spectrum with $\bar{\mathscr{L}}_t$. So, applying Theorem~3.9 in \cite{kato} for the operator $\tilde{\mathscr{L}_{t}}$ we obtain the families of eigenvalues and eigenfunctions with the required properties.
\end{proof}

We now consider a smooth map $\mathcal{F}:\mathcal{M}^r\rightarrow \mathcal{S}_+^2(M)$ that associates to each $g\in\mathcal{M}^r$ a symmetric positive definite tensor $\mathcal{F}(g)=T_{g}$. Naturally, $d\mathcal{F}_{g}$ stands for the differential of $\mathcal{F}$ at $g$, and $d\mathcal{F}_{g}^{\ast}$ for the adjoint of $d\mathcal{F}_{g}$. The next proposition constitutes the first Hadamard-type variational formula.
\begin{proposition}\label{pro bar-L}
Let $(M,g)$ be a compact Riemannian manifold and $g(t)$ be a smooth variation of $g$. Consider $\{\phi_i(t)\}\subset C^{\infty}(M)$ a differentiable family of functions and $\lambda_i(t)$ a differentiable family of real numbers
such that $\lambda_i(0)=\lambda$ for each $i=1,\ldots,m$ and for all $t$
\begin{equation*}
\left\{
\begin{array}{ccccc}
    -\bar{\mathscr{L}}_t\phi_i(t) &=& \lambda_i(t)\phi_i(t) & \mbox{in} & M\\
    \phi_i(t)&=&0 & \mbox{on} & \partial M,
\end{array}
\right.
\end{equation*}
with $\langle\phi_i(t),\phi_j(t)\rangle_{L^2(M,d\bar{m}_t)}=\delta_{ij}$. Then, we have
\begin{align}\label{eq.lm2}
    (\lambda_i+\lambda_j)'\delta_{ij} =& \int_{M}\big\langle\frac{1}{2}\bar{\mathscr{L}}(\phi_i\phi_j)g-2(T\nabla \phi_i)^\flat\otimes d\phi_j-2d\phi_i\otimes(T\nabla \phi_j)^\flat, H\big\rangle\dm\nonumber\\
     & +\int_M \Big[2\langle d\mathcal{F}_{g}^{\ast}\sym(d\phi_i\otimes d\phi_j) ,H\rangle +T(\nabla\dot{\eta},\nabla(\phi_{i}\phi_{j}))\Big]\dm,
\end{align}
where $(T\nabla\phi_i)^\flat$ is the $1$--form associated to $T\nabla\phi_i$ induced by $g$, and $\sym$ is the symmetrization operator.
\end{proposition}

\begin{proof}
Differentiating with respect to the variable $t$ the equation
\begin{equation*}
-\bar{\mathscr{L}}_{t}\phi_i(t)=\lambda_i(t)\phi_i(t)
\end{equation*}
at $t=0$, one has
\begin{equation*}
-\bar{\mathscr{L}}'\phi_i-\bar{\mathscr{L}}\phi'_i=\lambda_i'\phi_i+\lambda\phi'_i.
\end{equation*}
Multiplying by $\phi_j$ and using equation~\eqref{def.barL0}, we get 
\begin{equation*}
-\phi_j\mathscr{L}'\phi_i +\phi_jT(\nabla\dot{\eta},\nabla \phi_i) -\phi_j\bar{\mathscr{L}}\phi'_i = \lambda_i'\phi_j\phi_i-\phi_{i}'\bar{\mathscr{L}}\phi_{j}.
\end{equation*}
Using integration by parts, we obtain
\begin{equation*}
\lambda_i'\delta_{ij}= - \int_{M}\phi_j\mathscr{L}'\phi_i \dm +\int_{M} T(\nabla \dot{\eta}, \phi_j\nabla \phi_i)\dm.
\end{equation*}
So, we deduce from \eqref{eq.L'} that
\begin{align*}
(\lambda_i+\lambda_j)'\delta_{ij} =& -\int_{M}\phi_j\mathscr{L}'\phi_i \dm -\int_{M}\phi_i\mathscr{L}'\phi_j \dm +\int_{M} T(\nabla \dot{\eta}, \nabla(\phi_j\phi_i))\dm\\
=&-\int_{M}\!\! \Big[\frac{1}{2}T(\nabla h,\nabla(\phi_i\phi_j)) +\phi_j\dv_{\eta}(\mathscr{H}_{T}\nabla\phi_i) +\phi_i\dv_{\eta}(\mathscr{H}_{T}\nabla\phi_j)\Big]\dm\\
& +\int_{M} T(\nabla \dot{\eta}, \nabla(\phi_j\phi_i))\dm.
\end{align*}
Again, integration   by parts implies
\begin{equation}\label{eq.LB}
(\lambda_i+\lambda_j)'\delta_{ij} =\int_{M}\left(\frac{h}{2}\bar{\mathscr{L}}(\phi_i\phi_j) +2\mathscr{H}_{T}(\nabla\phi_i,\nabla\phi_j) +T(\nabla \dot{\eta}, \nabla(\phi_j\phi_i))\right)\dm.
\end{equation}
Now, we compute
\begin{eqnarray*}
&&\mathscr{H}_{T}(\nabla\phi_i,\nabla\phi_j)\\
&=&-T(H\nabla\phi_i,\nabla\phi_j)-H(T\nabla\phi_i,\nabla\phi_j)+\langle T',d\phi_i\otimes d\phi_j \rangle\\
&=&\langle-(T\nabla \phi_i)^\flat\otimes d\phi_j- d\phi_i\otimes(T\nabla \phi_j)^\flat, H\rangle+\langle d\mathcal{F}_g(H),\sym(d\phi_i\otimes d\phi_j) \rangle\\
&=&\langle-(T\nabla \phi_i)^\flat\otimes d\phi_j- d\phi_i\otimes(T\nabla \phi_j)^\flat +(d\mathcal{F}_g)^\ast\sym(d\phi_i\otimes d\phi_j), H\rangle.
\end{eqnarray*}
We substitute this into~\eqref{eq.LB} to complete the proof.
\end{proof}

We point out that Hadamard-type variation formula~\eqref{eq.lm2} generalizes the Berger formula in \cite{berger}. Moreover, we highlight that equation \eqref{eq.LB} will be useful in different ways in this paper.

\section{On generic properties of eigenvalues}\label{Section GenDirc}
Here, we give an application of the Hadamard-type formula that we have obtained in the previous section.

Let $\mathcal{F}:\mathcal{M}^r\rightarrow \mathcal{S}_+^2(M)$ be a smooth map that associates to each $g\in\mathcal{M}^r$ a symmetric positive definite tensor $\mathcal{F}(g)=T_{g}$ such that $\mathscr{L}_g(\cdot)=\dv_{\eta}(T_g\nabla\cdot)$ is uniformly elliptic. We say that the $T_g$--family satisfies the property $\mathcal{P}$ if, for all $g\in\mathcal{M}^r$, the tensor $G_g=(n-4)T_g +2d\mathcal{F}_g(g)$ satisfies either $G_g(X,X)> 0$ or $G_g(X,X)<0$ or $G_g(X,X)\equiv 0$ for all $X\in\mathfrak{X}(M)$. For example, if $\psi$ is a positive smooth function then $T_g=\psi g$ satisfies the property $\mathcal{P}$. 

\begin{theorem}\label{thm1}
Let $(M,g_0)$ be a compact Riemannian manifold. Consider a family of operators $\mathscr{L}_g(\cdot)=\dv_{\eta}(T_g\nabla\cdot)$ where the $T_g$--family satisfies the property $\mathcal{P}$, and let $\lambda$ be an eigenvalue of the problem
\begin{equation}\label{problem.var metric}
\left\{
\begin{array}{ccccc}
-\mathscr{L}_{g_0}\phi&=& \lambda\phi &\hbox{in}& M\\
\phi&=&0 &\hbox{on}& \partial M
\end{array}
\right.
\end{equation}
with multiplicity $m>1$. Then given any neighborhood $U\subset\mathcal{M}^r$ of $g_0$ and $\epsilon> 0$, there exists $g\in U$ and a simple eigenvalue $\lambda(g)$ such that  $|\lambda(g)-\lambda(g_0)|< \epsilon$. In particular, the subset $\Gamma\subset \mathcal{M}^r$ where the eigenvalues of \eqref{problem.var metric} are simple, is residual. 
\end{theorem}
\begin{proof}
We will argue by contradiction. Let $U\subset\mathcal{M}^r$ be a neighborhood of $g_0$ such that $\forall g\in U$ with $|\lambda(g)-\lambda(g_0)|< \epsilon$  the multiplicity of $\lambda(g)$ is $m$. Consider $g(t)=g_0+tH$, where $H$ is any symmetric $(0,2)$--tensor on $(M^n,g(t))$ with  $\lambda(t)$ and $\phi(t)$ given by Lemma \ref{LemExist} such that
\begin{equation*}
\left\{
\begin{array}{ccccc}
-\mathscr{L}_t\phi_i(t)&=& \lambda(t)\phi_i(t) &\hbox{in}& M\\
\phi_i(t)&=&0 &\hbox{on}& \partial M.
\end{array}
\right.
\end{equation*}
By Proposition \ref{pro bar-L}
\begin{align*}
\lambda'\delta_{ij} =& \int_{M}\langle\frac{1}{4}\mathscr{L}(\phi_i\phi_j)g_0-\sym((T\nabla \phi_i)^\flat\otimes d\phi_j), H\rangle\dm\nonumber\\
     &+\int_M \langle-\sym(d\phi_i\otimes(T\nabla \phi_j)^\flat)+ d\mathcal{F}_{g_0}^{\ast}\sym(d\phi_i\otimes d\phi_j) ,H\rangle\dm.
\end{align*}
If $i\neq j$, we have
\begin{eqnarray}\label{eq S}
0&=&\frac{1}{4}\mathscr{L}(\phi_i\phi_j)g_0-\sym((T\nabla \phi_i)^\flat\otimes d\phi_j)-\sym(d\phi_i\otimes(T\nabla \phi_j)^\flat)\nonumber\\
&&+ d\mathcal{F}_{g_0}^{\ast}\sym(d\phi_i\otimes d\phi_j).
\end{eqnarray}
Furthermore, taking traces in equation \eqref{eq S}, we get
\begin{eqnarray}\label{eq-aux thm1}
0&=&n\mathscr{L}(\phi_i\phi_j)-8T(\nabla\phi_i,\nabla\phi_j)+ 4d\mathcal{F}_{g_0}(g_0)(\nabla\phi_i,\nabla\phi_j)\nonumber\\
&=&-2\lambda n\phi_i\phi_j+2\big((n-4)T +2d\mathcal{F}_{g_0}(g_0)\big)(\nabla\phi_i,\nabla\phi_j).
\end{eqnarray}
We so obtain
\begin{equation}\label{UI}
\lambda n\phi_i\phi_j=G_g(\nabla\phi_i,\nabla\phi_j).
\end{equation}
Motivated by a similar argument given by Uhlenbeck~\cite{uhlenbeck} we fix $p\in M$ and consider an integral curve $\alpha$ in $M$ such that $\alpha(0)=p$ and $\alpha'(s)=G_g\nabla\phi_i(\alpha(s))$. Defining $\beta(s):=\phi_j(\alpha(s))$, we compute
\begin{eqnarray*}
\beta'(s)&=&\langle \nabla\phi_j(\alpha(s)),\alpha'(s)\rangle =G_g(\nabla\phi_j,\nabla\phi_i)(\alpha(s))=\lambda n\phi_{i}(\alpha(s))\phi_{j}(\alpha(s))\\
&=&\lambda n\phi_{i}(\alpha(s))\beta(s).
\end{eqnarray*}
This gives us that $\beta(s)=ce^{n\lambda\int_{0}^{s}\phi_i(\alpha(t))dt}$, where $c>0.$ If $G_g$ is positive, we have that $\phi_i(\alpha(s))$ is increasing in $s$ and then $\beta(s)\nearrow\infty$, which is a contradiction, since $M$ is compact. If $G_{g}$ is negative, the argument is analogous. If $G_g\equiv0$, then from equation~\eqref{eq-aux thm1} we get $\phi_i\phi_j=0$, hence, by the principle of the unique continuation (see~\cite{homander}) we have that at least one of the eigenfunctions vanishes, which is again a contradiction. It proves the first part of the theorem.

Now, let $\Gamma_m$ be the set of metrics $g\in\mathcal{M}^r$ such that the first $m$ eigenvalues of $\mathscr{L}_g$ with respect to Problem \eqref{problem.var metric} are simple. It is known that those eigenvalues depend continually of the metric (see \cite{bando}), hence $\Gamma_m$ is open in $\mathcal{M}^r$. On the other hand, it follows from the first part of the theorem that $\Gamma_m$ is dense in $\mathcal{M}^r$. Since $\mathcal{M}^r$ is a complete metric space on the topology
$\mathcal{C}^r$, the set $\Gamma = \cap_{m=1}^{\infty}\Gamma_m$ is residual.
\end{proof}

\section{Neumann boundary eigenvalue problem}\label{Section Nbc}

Our goal in this section is to prove the generic properties for eigenvalues of the operator $\mathscr{L}$ with the Neumann boundary condition. More specifically,
\begin{equation}\label{Neumann-plrp}
\left\{\begin{array}{ccccc}
(\mathscr{L}+ \lambda)\phi &=& 0& \text{in}& M\\
T(\nabla\phi,\nu) &=& 0 & \text{on}& \partial M,
\end{array}\right.
\end{equation}
where $\nu$ is the outward pointing unit normal vector field along $\partial M$. 
\subsection{Hadamard-type variation formulas}

Here, we proceed as in Section~\ref{Section Dbc}. Let us consider a smooth family of Riemannian metrics $g(t)$ and $T_t=T_{g(t)}$. 
\begin{proposition}\label{Neumann-prop3} The following holds for any $ f,\ell\in C^{\infty}(M).$
\begin{equation*}
\int_{M}\ell\bar{\mathscr{L}}' f \dm =\int_M \ell\Big(\frac{1}{2}T(\nabla h,\nabla f)+ \dv_\eta(\mathscr{H}_T\nabla f)-T(\nabla\dot{\eta},\nabla f)\Big)\dm,
\end{equation*}
where $\bar{\mathscr{L}}':=\frac{d}{dt}\big|_{t=0}\bar{\mathscr{L}}_{g(t)}.$
\end{proposition}

\begin{proof}
As in the proof of Lemma~\ref{lemEq.bar.L'}, we begin by considering $\eta$ to be independent of $t$. In this case, $\bar{\mathscr{L}}_t$ becomes $\mathscr{L}_t$. By integration by parts
\begin{equation*}
\int_{M}\ell\mathscr{L}_t f \dmt=-\int_{M}T(\nabla f,\nabla\ell) \dmt + \int_{\partial M}\ell T(\nabla f,\nu_t)\dnt.
\end{equation*}
Thus, from \eqref{Neumann-eq3lem2} at $t=0$,
\begin{align*}
\int_M\ell\mathscr{L}' f \dm + \frac{1}{2}\int_M\ell h\mathscr{L}  f \dm=&-\int_M\mathscr{H}_T(\nabla  f,\nabla \ell)\dm -\frac{1}{2}
\int_MhT(\nabla f,\nabla\ell) \dm\\
&+\int_{\partial M}\ell\Big(\mathscr{H}_T(\nu,\nabla  f)+\frac{1}{2}H(\nu,\nu)T(\nabla f,\nu)\Big)\dn\\
&+\int_{\partial M}\ell\frac{\tilde{h}}{2}T(\nabla  f,\nu) \dn.
\end{align*}
 Rearranging the above equation, we have
\begin{align}\label{Neumann-eq2lem1}
\nonumber\int_M\ell\mathscr{L}' f \dm =&-\int_M\mathscr{H}_T(\nabla  f,\nabla \ell)\dm +\int_{\partial M}\ell \mathscr{H}_T(\nu,\nabla f)\dn \\
&-\frac{1}{2}\int_M\big(hT(\nabla f,\nabla\ell)+\ell h\mathscr{L}  f\big)\dm\nonumber\\
&+\frac{1}{2}\int_{\partial M}\ell\big(\tilde{h}+H(\nu,\nu)\big)T(\nabla f,\nu) \dn.
\end{align}
Finally, plugging \eqref{eq3lem1} and \eqref{eq4lem1} in \eqref{Neumann-eq2lem1}, we obtain
\begin{align*}
\int_{M}l\mathscr{L}' f \dm =&\int_M \ell\Big(\frac{1}{2}T(\nabla h,\nabla f)+\dv_\eta(\mathscr{H}_T\nabla f)\Big)\dm \\
& + \frac{1}{2}\int_{\partial M}\ell\big(-h + H(\nu,\nu)+\tilde{h}\big)T(\nabla f,\nu) \dn,
\end{align*}
for all $\ell\in C^{\infty}(M)$, which is sufficient for completing the proof of our proposition, since $\tilde{h}=\mathrm{tr}_g(H|_{\partial M})=
h - H(\nu,\nu)$. For the general case, we use equation~\eqref{def.barL0} to conclude the proof.
\end{proof}

Below we will obtain Hadamard-type formulas for the eigenvalues of $\mathscr{L}$ with $T_g$--Neumann boundary condition.
\begin{proposition}\label{Neumann-pro1}
Let $(M,g_0)$ be a compact Riemannian manifold, and $g(t)$ be a smooth variation of $g_0$. Consider $\{\phi_i(t)\}\subset C^{\infty}(M)$ a differentiable family of functions and $\lambda_i(t)$ a differentiable family of real numbers such that $\lambda_i(0)=\lambda$ for each $i=1,\ldots,m$ and for all $t$
\begin{equation*}
\left\{
  \begin{array}{ccccc}
    -\bar{\mathscr{L}}_{g(t)}\phi_{i}(t) &=& \lambda_i(t)\phi_{i}(t) & \text{in}& M\\
    T_t(\nu_t,\nabla\phi_{i}(t))&=&0 & \text{on}& \partial M,
  \end{array}
\right.
\end{equation*}
which $\langle\phi_i(t),\phi_j(t)\rangle_{L^2(M,\dm_t)}=\delta_{ij}$. Then, we obtain the following variation formula
\begin{align*}
(\lambda_i+\lambda_j)'\delta_{ij} =& \int_{M}\langle\frac{1}{2}\mathscr{L}_{g_0}(\phi_i\phi_j)g_0-2(T_{g_0}\nabla \phi_i)^\flat\otimes d\phi_j- 2d\phi_i\otimes
(T_{g_0}\nabla \phi_j)^\flat, H\rangle\dm\nonumber\\
& +\int_M \Big[2\langle d\mathcal{F}_{g_0}^{\ast}\sym(d\phi_i\otimes d\phi_j) ,H\rangle + T(\nabla \dot{\eta}, \nabla (\phi_i\phi_j))\Big]\dm,
\end{align*}
where $(T_{g_0}\nabla\phi_i)^\flat$ is the $1$-form associated to $T_{g_0}\nabla\phi_i$ induced by $g_0$, and $\sym$ denotes the symmetrization operator.
\end{proposition}

\begin{proof}
Taking the derivative with respect to $t$ in both sides of the identity
\begin{equation*}
-\bar{\mathscr{L}}_{g(t)} \phi_{i}(t)=\lambda_i(t)\phi_{i}(t),
\end{equation*}
one has at $t=0$
\begin{equation*}
-\bar{\mathscr{L}}'\phi_i-\bar{\mathscr{L}}\phi'_i=\lambda_i'\phi_i+\lambda\phi'_i.
\end{equation*}
Thus
\begin{equation*}
-\int_{M}(\phi_j\bar{\mathscr{L}}'\phi_i+\phi_j\bar{\mathscr{L}}\phi'_i)\dm
=\int_{M}(\lambda_i'\phi_j\phi_i-\phi_i'\bar{\mathscr{L}}\phi_j)\dm.
\end{equation*}
Now, since $T(\nu_t,\nabla_t\phi_{i}(t))=0$ on $\partial M$, we deduce from~\eqref{Neumann-eq3lem2} that
\begin{equation*}
T(\nu,\nabla\phi'_i)=-\mathscr{H}_T(\nu,\nabla\phi_i).
\end{equation*}
Moreover, integration by parts gives
\begin{align*}
\lambda_i'\delta_{ij}=& - \int_{M}\phi_j\bar{\mathscr{L}}'\phi_i \dm - \int_{\partial M}\phi_jT(\nu,\nabla\phi_i')\dn\\
 =& - \int_{M}\phi_j\bar{\mathscr{L}}'\phi_i \dm + \int_{\partial M}\phi_j\mathscr{H}_T(\nu,\nabla\phi_i)\dn.
\end{align*}
Whence,
\begin{align*}
-(\lambda_i+\lambda_j)'\delta_{ij} =& \int_{M}\phi_j\bar{\mathscr{L}}'\phi_i \dm + \int_{M}\phi_i\bar{\mathscr{L}}'\phi_j \dm - \int_{\partial M}\phi_i\mathscr{H}_T
(\nu,\nabla\phi_j)\dn \\
&- \int_{\partial M}\phi_j\mathscr{H}_T(\nu,\nabla\phi_i)\dn\\
 =&\int_{M} \frac{1}{2}T(\nabla h,\nabla(\phi_i\phi_j))\dm + \int_{M}\phi_i\dv_\eta(\mathscr{H}_T\nabla\phi_j)\dm\\
 &+\int_{M}\phi_j\dv_\eta (\mathscr{H}_T\nabla\phi_i)\dm +\int_{M} T(\nabla \dot{\eta}, \nabla (\phi_i\phi_j))\dm\\
 &- \int_{\partial M}\phi_i\mathscr{H}_T(\nu,\nabla\phi_j)\dn - \int_{\partial M}\phi_j\mathscr{H}_T(\nu,\nabla\phi_i)\dn.
\end{align*}
Next, we use the divergence theorem to compute
\begin{equation}\label{eq.LB-TgN}
(\lambda_i+\lambda_j)'\delta_{ij}= \int_{M}\frac{h}{2}\bar{\mathscr{L}}(\phi_i\phi_j)\dm + 2 \int_{M}\mathscr{H}_T(\nabla\phi_i,\nabla\phi_j) \dm+\int_{M} T(\nabla \dot{\eta}, \nabla (\phi_i\phi_j))\dm
\end{equation}
or equivalently
\begin{align*}
(\lambda_i+\lambda_j)'\delta_{ij} =& 2\int_{M}\langle\frac{1}{4}\mathscr{L}(\phi_i\phi_j)g-(T\nabla \phi_i)^\flat\otimes d\phi_j- d\phi_i\otimes
(T\nabla \phi_j)^\flat, H\rangle\dm\nonumber\\
& +2\int_M \langle d\mathcal{F}_g^{\ast}\sym(d\phi_i\otimes d\phi_j) ,H\rangle\dm+\int_{M} T(\nabla \dot{\eta}, \nabla (\phi_i\phi_j))\dm
\end{align*}
as in Proposition \ref{pro bar-L}.
\end{proof}

Now, we will prove the existence of analytic curves of eigenvalues.
\begin{proposition}\label{Neumann-thmExistence}
Let $(M, g_0)$ be a compact Riemannian manifold and $g(t)$ be a real analytic one-parameter family of Riemannian metrics on $M$ with $g(0)=g_0$. Assume $\lambda$ is an eigenvalue of multiplicity $m$ for the operator $\mathscr{L}_{g_0}$ with $T_{g_0}$--Neumann boundary condition. Then, there exist $\varepsilon>0$, and $t$-analytic functions $\lambda_{i}(t)$ and
$\phi_{i}(t)$, with $i=1,\dots,m$ and $|t|<\varepsilon$, such that  $\langle\phi_{i}(t),\phi_{j}(t)\rangle_{L^2(M,\dmt)}=\delta_i^j$,  $\lambda_{i}(0)=\lambda$ and
\begin{equation*}
\left\{
\begin{array}{cccl}
-\mathscr{L}_t\phi_{i}(t) &=&\lambda_{i}(t)\phi_{i}(t)  & \text{in } M\\
T_t(\nabla\phi_i(t),\nu_t)  &=& 0 & \text{on } \partial M.
\end{array}
\right.
\end{equation*}
\end{proposition}

The following remark is due at this point. The main technique utilized in the proof of Proposition~\ref{Neumann-thmExistence} is the Liapunov-Schmidt method. In particular, we will follow along the same lines of proof as in Henry \cite{henry}, Marrocos and Pereira~\cite{marrocos}, and Gomes and Marrocos~\cite{gomes}, for they already successfully employed this method in contexts similar to ours. The proof of Proposition~\ref{Neumann-thmExistence} requires the following proposition.

\begin{proposition}\label{Neumann-tcaln}
Let $(M , g_0)$ be a compact Riemannian manifold, and $\lambda_0$ be an eigenvalue of multiplicity $m>1$ of the operator $\mathscr{L}$ with $T_{g_0}$--Neumann boundary condition. Then for every $\epsilon>0$ there is $\delta>0$ so that for each
$|t|<\delta$, there exist exactly $m$ eigenvalues (counting their multiplicities) to Problem~\eqref{Neumann-plrp} in the interval
$(\lambda_0-\epsilon,\lambda_0+\epsilon)$.
\end{proposition}
\begin{proof}
Let $\{\phi_j\}_{j=1}^m$ be an orthonormal basis associated with $\lambda_0$, and consider the orthonormal projection on the corresponding eigenspace
\[
Pu=\sum_{j=1}^m\phi_j\int_M\phi_j u \dg.
\]
Note that $P$ induces the splitting $L^2(M,\dg)={\mathcal{R}}(P)\oplus\mathcal{N}(P)$
so that any function $u\in L^2(M,\dg)$ can be written as $u=\phi+\psi$, where  $\phi\in\mathcal{R}(P)=ker(\mathscr{L}+\lambda_0)$ and $\psi\in \mathcal{N}(P)$. With this in mind, we consider the family of metrics $g(t)$ on $M$ and the family of eigenvalue problems, where the operator $\mathscr{L}$ remains fixed and the families $T_t$ and $\nu_t$ change with the parameter $t$ along the boundary,
\begin{equation}\label{Neumann-deqlr}
\left\{\begin{array}{ccccc}
(I-P)(\mathscr{L} + \lambda)(\phi+\psi)    &=& 0& \text{in}&  M\\
P(\mathscr{L}+ \lambda)(\phi+\psi)         &=& 0& \text{in}&  M\\
T_t (\nu_t,\nabla(\phi+\psi)) &=& 0& \text{on}& \partial M.
\end{array}\right.
\end{equation}

To solve problem \eqref{Neumann-deqlr}, we first observe that since $\phi_j$ and $\psi$ are orthonormal, by the divergence theorem we must have
\begin{eqnarray*}
P(\mathscr{L} + \lambda)\psi  &=&  \sum_{j=1}^m\phi_j\int_{M}\phi_j(\mathscr{L} + \lambda)\psi\dg\\
&=& \sum_{j=1}^p\phi_j \int_{M}\phi_j(\mathscr{L}+\lambda)\psi-\psi(\mathscr{L}+\lambda)\phi_j\dg   \\
&=& \sum_{j=1}^p\phi_j \int_{\partial M}\phi_jT(\nabla\psi,\nu) -\psi T(\nabla\phi,\nu)\db \\
&=& \sum_{j=1}^m\phi_j\int_{\partial M}\phi_j T(\nabla\psi,\nu) \db
\end{eqnarray*}
which implies
\begin{equation*}
(\mathscr{L}+\lambda)\psi  =  (I-P)\big((\mathscr{L}+\lambda)\psi\big)+\displaystyle\sum_{j=1}^m\phi_j\int_{\partial M}\phi_j T(\nabla\psi,\nu) \db.
\end{equation*}
Thus, we get
\begin{equation*}
(\mathscr{L}+ \lambda)\psi+ (I-P)(\mathscr{L}_t - \mathscr{L})(\phi+\psi)-\sum_{j=1}^m\phi_j\int_{\partial M}\phi_jT(\nabla\psi,\nu)\db =0.
\end{equation*}
Moreover, the part concerning the boundary in \eqref{Neumann-deqlr} can be rewritten as
\begin{equation*}
T(\nu,\nabla \psi) +T_t(\nu_t,\nabla(\phi+\psi))-T(\nu,\nabla(\phi+\psi))=0.
\end{equation*}
Hence, to solve the first and the third equations of \eqref{Neumann-deqlr}, is equivalent to finding the zeros of the map
\begin{eqnarray*}
F :  \mathbb{R}\times\mathbb{R}\times\mathcal{R}(P)\times H^2(M)\cap\mathcal{N}(P)&\longrightarrow &\mathcal{N}(P)\times H^{\frac{3}{2}}(M)\\
(t,\lambda,\phi,\psi)&\mapsto&\big(F_1(t,\lambda,\phi,\psi),F_2(t,\lambda,\phi,\psi)\big),
\end{eqnarray*}
where
\begin{equation*}
\left\{\begin{array}{lcc}
F_1=(\mathscr{L} + \lambda)\psi+ (I-P)(\mathscr{L}_t - \mathscr{L})(\phi+\psi)-\displaystyle\sum_{j=1}^m\phi_j\int_{\partial M}\phi_jT(\nabla\psi,\nu)\db\\
F_2 = T(\nu,\nabla \psi) +T_t(\nu_t,\nabla(\phi+\psi))-T(\nu,\nabla(\phi+\psi)).
\end{array}\right.
\end{equation*}
Note that $F$ depends differentially on the variables $\lambda$, $t$, $\psi$ e $\phi$. Our intention is to use the implicit function theorem to show that
$F(t,\lambda,\phi,\psi)=(0,0)$ admits a solution $\psi$ as function of $\lambda$, $t$ and $\phi$. To this end, we observe that if $t = 0, \lambda=
\lambda_0$ and $\psi=0$, then
\begin{equation}\label{Neumann-TAI}
\frac{\partial F}{\partial \psi}(0,\lambda_0 , 0 , 0 )\dot{\psi}= \Big((\mathscr{L} + \lambda_0) \dot{\psi} - \sum_{j=1}^m \phi_j \int_{\partial M}\phi_j
T(\nabla\dot{\psi},\nu)\db \, ,\, \frac{\partial \dot{\psi} }{\partial \nu}\Big).
\end{equation}
We claim now that the map given in \eqref{Neumann-TAI} is an isomorphism from $H^2(M)\cap\mathcal{N}(P)$ onto $\mathcal{N}(P)\times H^{\frac{3}{2}}(M)$.
Indeed, the proof of this fact can be found in \cite{lions}.

Hence, by the implicit function theorem there exist positive numbers $\delta$, $\epsilon$ and a function $S(t , \lambda)\phi$ of class $\mathcal{C}^1$ at the variables
$(t , \lambda)$ such that for every $|t| < \delta$ and  $\lambda \in (\lambda_0-\epsilon, \lambda_0+\epsilon)$, $F(t ,\lambda,\phi, S(t,\lambda)\phi)=
(0,0)$. Furthermore, $S(t , \lambda)\phi$ is analytic at $\lambda$ and linear at $\phi$. This solves the equation \eqref{Neumann-deqlr} in relation to $\psi$.

We now observe that for every  $\phi\in\mathcal{R}(P)$ there exist real numbers $c_1,c_2,\dots,c_m$ so that $\phi = \sum_{j=1}^{m}c_j\phi_j$. Thus, the second equation in \eqref{Neumann-deqlr} can be equivalently perceived as a system of equations in the variables  $c_1,\dots,c_m$ as below
\begin{equation*}
\sum_{j=1}^m c_j\int_{M}\phi_k (\mathscr{L}_t + \lambda)(\phi_j+S(t,\lambda)\phi_j)\dg=0,\quad k=1,2,\dots, m.
\end{equation*}
In this way, $\lambda$ is an eigenvalue of $\mathscr{L}_t$ if and only if $\det A(t,\lambda)=0$, where $A(t,\lambda)$ is given by
\begin{equation*}
A_{kj}(t,\lambda)=\int_{M}\phi_k(\mathscr{L}_t + \lambda) (\phi_j+S(t,\lambda)\phi_j)\dg.
\end{equation*}
Furthermore, the associated eigenfunctions are given by
\begin{equation*}
u(t, \lambda)=\sum_{j=1}^{m}c_j(\phi_j+S(t,\lambda)\phi_j).
\end{equation*}
In other words, $c=(c_1,\ldots,c_m)$ must satisfy $A(t,\lambda) c=0$. It turns out that by Rouch\'e theorem, we have that: For every $\epsilon>0$
there is $\delta > 0$ so that if $|t - t_0| < \delta $, then there exist exactly $m$-roots of $\det A(t ,\lambda)=0$ in the interval
$(\lambda_0-\epsilon,\lambda_0+\epsilon)$.
\end{proof}

\begin{proof}[{\rm\textbf{Proof of Proposition \ref{Neumann-thmExistence}.}}]
Assume the same hypotheses as in Proposition~\ref{Neumann-tcaln}. We must show that there exist $m$ analytic curves of eigenvalues $\lambda_j(t)$ for $(\ref{Neumann-plrp})$ associated to $m$ analytic curves of eigenfunctions  $\phi_j(t)$. The strategy of proof is to reduce the problem to a finite-dimensional one and to apply Kato's Selection theorem, see \cite{kato}. With this in mind, we will make a slightly different construction than that of Proposition~\ref{Neumann-tcaln}. This will result, as we will see shortly, in obtaining a symmetric matrix.

Let $\{\phi_j\}_{j=1}^m$ be orthonormal eigenfunctions of the Laplace-Neumann operator associated with $\lambda_0$. For each $j=1,\ldots, m$ consider the following problem
\begin{equation}\label{Neumann-pfp}
\left\{\begin{array}{ccccc}
(\mathscr{L} +\lambda_0)u   &=0& \text{in}&  M\\
T_t(\nu_t,\nabla(\phi_j+ u)) &=0& \text{on}& \partial M\\
Pu=\displaystyle\sum_{j=1}^m\phi_j\int_{M}\phi_ju\dg &=0& \text{in}& M.
\end{array}\right.
\end{equation}

Consider now the orthogonal complement $[\phi_j]^{\bot}$ of $ker(\mathscr{L}+\lambda_0)$ in $L^2(M,\dm_0)$ and define
\begin{equation*}
F: (-\delta, \delta)\times H^2(M,\dm_0)\longrightarrow [\phi_j]^{\bot}\times\mathcal{R}(P)\times H^{\frac{3}{2}}(M,\dm_0)
\end{equation*}
by
\begin{equation*}
F(t,w)=\big((\mathscr{L}+\lambda_0)w,\, Pw,\, T_t(\nu_t,\nabla (\phi_j+w))\big).
\end{equation*}
Exactly as before we get that $\frac{\partial F}{\partial w}(0,0)$ is an isomorphism, so by the implicit function theorem there exist $\delta > 0$ and an analytic
function $w_j(t)$ defined on $|t - t_0| < \delta$ such that $F(t,w_j(t)) = 0$. In addition, we obtain for each $|t - t_0| < \delta$ a linearly independent set
of functions $\{\varphi_j(t)\}_{j=1}^m$, given by $\varphi_j(t)=\phi_j+w_j(t)$,
that satisfy the equation $(\ref{Neumann-pfp})$. By using the Gram-Schmidt orthonormalization process with respect to the inner product
\[
(u,v):=\int_{M}uv\, \dmt,
\]
we can without loss of generality assume  that $\{\varphi_j(t)\}_{j=1}^m$ is orthonormal. Note that the functions $\varphi_j(t)$ belong to $D_t=
\{u \in H^{2}(M,\dm_0), T_t(\nu_{t},\nabla u) = 0\}$. Moreover, since $\mathscr{L}_t$ is self-adjoint with respect to the inner product
defined above, it follows that the matrix $\int_{M}\varphi_j \mathscr{L}_t \varphi_k \dmt$ is symmetric.

For a given $\mathcal{T}\in \mathcal{S}^{2,k}$, we define a family of Riemannian metrics on $M$ by $g(t) = g_0 +t\mathcal{T}$ and let $P(t)$ be given by
\begin{equation*}
P(t)u=\sum_{j=1}^m\varphi_j(t)\int_{M}u\varphi_j(t)\dmt.
\end{equation*}
We finally define for each $j=1,\dots, m$,
\begin{eqnarray*}
 G_j :  (-\epsilon, \epsilon)\times\mathbb{R}\times H^2(M)&\longrightarrow & [\phi_j]^{\bot}\times H^{\frac{3}{2}}(M,\dm_0)\times\mathcal{R}(P)\\
 (t,\lambda,w)&\mapsto & \big(G_{j1}(t,\lambda,w),G_{j2}(t,\lambda,w),G_{j3}(t,\lambda,w)\big)
\end{eqnarray*}
by
\begin{equation*}
\left\{\begin{array}{lcc}
G_{j1}=(I-P(t))((\mathscr{L}_t + \lambda))(w+\varphi_j(t))\\
G_{j2}=T_t(\nu_t,\nabla w);\\
G_{j3}= P(t)w.\end{array}\right.
\end{equation*}
Once again, the implicit function theorem yields a number   $\delta>0$ and functions $w_j(t,\lambda)$ such that for any $|t- t_0|<\delta$ and every
$|\lambda-\lambda_0|<\delta$, the equality $G_j(t,\lambda,w_j(t,\lambda))=(0,0,0)$ holds. As we know,  $\lambda$  is an eigenvalue for \eqref{Neumann-plrp}
if and only if there exists a nonzero $m$-uple $c=(c_1,\ldots,c_m)$ of real numbers such that $A(t,\lambda)c=0$, where
\begin{equation*}
A_{ij}(t,\lambda)= \int_{M}\varphi_i(t)(\mathscr{L}_t+\lambda)(\varphi_j(t)+w_j(t,\lambda))\dmt.
\end{equation*}
That is, $\lambda$ is an eigenvalue of \eqref{Neumann-plrp} if $\det A(t,\lambda)=0$. By Rouch\'e's theorem, there exist $m$ roots near $\lambda_0$
counting multiplicity, for each $t$. So by Puiseux's theorem \cite{wall} there exist $m$-analytic functions $t\to \lambda_i(t)$ which locally solve the
equation $\det A(t,\lambda)=0$. It can be easily seen that  $A$ is symmetric and hence, by Kato's Selection theorem~\cite{kato}, we can find an analytic
curve $c^i(t)\in\mathbb{R}^m$ such that $A(t,\lambda_i(t))c^i(t)=0$, for each $i=1,\dots,m$. Thus $$\psi_i(t)=\sum_{j=1}^m c^i_j(t)(\varphi_j
+\omega_j(t,\lambda_i(t)))$$ is an analytic curve of eigenfunctions for \eqref{Neumann-plrp} associated with $\lambda_i(t)$. Reasoning exactly as Kato in
\cite[p.~98]{kato} we can obtain $m$-analytic curves of eigenvalues $\{\phi_i(t)\}_{i=1}^m$ such that $\int_{M}\phi_i(t)\phi_j(t)\dmt=\delta_i^j$.
\end{proof}

\begin{remark}
In the case of $m=m(\lambda_0)=1$, the existence of a differentiable curve of eigenvalues through $\lambda_0$ follows directly from the implicit function
theorem applied to the map $F:S^k\times H^2(M,\dm_0)\times \mathbb{R}\rightarrow L^2(M, \dm_0)\times \mathbb{R}$ defined by
\begin{equation*}
F(g, u , \lambda)= \Big((\mathscr{L}_g + \lambda)u,\int_{M}u^2\dg\Big).
\end{equation*}
The corresponding formula to the derivative $\lambda'(t)$ can be obtained by letting $i=j=1$ in Proposition \ref{Neumann-pro1}.
\end{remark}

\begin{remark}
We observe that Theorem~\ref{thm1} remains valid if we replace the Dirichlet boundary condition by the $T_g$--Neumann boundary condition.
\end{remark}

\section{Domain variation}\label{DV}
In this section, we consider the case of domain deformation. For this, let $(M,g)$ be a Riemannian manifold, $\Omega\subset M$ be a bounded domain with smooth boundary $\partial\Omega$,  and $T$ be a symmetric $(0,2)$--tensor on $M$. We consider a deformation $\Omega_t$ by the family of diffeomorphisms $f_t:\Omega\to\Omega_t$, where $f_0=id_\Omega$ is the identity map. The smooth variations of $g$ and $T$ are given by
\begin{equation*}
g(t)=f_{t}^{\ast}g\quad  \mbox{and} \quad T_t=f_{t}^{\ast}T.
\end{equation*}
We denote
\begin{equation*}
V=\frac{d}{dt}\big|_{t=0}f_t,\quad H=\frac{d}{dt}\big|_{t=0}f_t^\ast g=\mathcal{L}_Vg \quad \mbox{and} \quad
T'=\frac{d}{dt}\big|_{t=0}f_t^\ast T=\mathcal{L}_VT.
\end{equation*}

The next two technical lemmas are necessary tools to obtain the Hadamard-type formulas given by Proposition~\ref{Hadamard for domains}. 
\begin{lemma}
For $X,Y\in\mathfrak{X}(M)$, we have
\begin{equation}\label{eq.scrH.XY}
\mathscr{H}_T(X,Y)=-\langle \nabla_{TX}V,Y\rangle-\langle X,\nabla_{TY}V\rangle+\langle(\nabla_VT)X,Y \rangle.
\end{equation}
\end{lemma}
\begin{proof}
\begin{eqnarray}
T'(X,Y)&=& (\mathcal{L}_VT)(X,Y)\nonumber\\
       &=& V\langle TX,Y\rangle-T(\nabla_VX-\nabla_XV,Y)-T(X,\nabla_VY-\nabla_YV)\nonumber\\
       &=& T(X,\nabla_YV)+T(\nabla_XV,Y)+\langle \nabla_V(TX),Y\rangle-T(\nabla_VX,Y)\nonumber\\
       &=& T(X,\nabla_YV)+T(\nabla_XV,Y)+\langle (\nabla_V T)X,Y\rangle.\nonumber
\end{eqnarray}
In particular, for $T=g$, we get
\begin{equation*}
H(X,Y)=\langle\nabla_XV,Y \rangle+\langle X,\nabla_YV \rangle.
\end{equation*}
Follows that
\begin{equation*}
HT(X,Y)=H(TX,Y)=\langle\nabla_{TX}V,Y \rangle+\langle TX,\nabla_YV \rangle
\end{equation*}
and
\begin{equation*}
TH(X,Y)=H(X,TY)=\langle\nabla_XV,TY \rangle+\langle X,\nabla_{TY}V \rangle.
\end{equation*}
By using the identity $\mathscr{H}_T=T'-(HT+TH)$ we conclude our proof.
\end{proof}

\begin{lemma}\label{lem.domH} For any eigenfunctions associated to the eigenvalue $\lambda$ we have
\begin{equation}
  \mathscr{H}_T(\nabla\phi_i,\nabla\phi_j)=-\dv_{\eta}(\langle V,\nabla\phi_{j}\rangle T\nabla\phi_{i}+\langle V,\nabla\phi_{i}\rangle T\nabla\phi_{j})+\frac{1}{2}\langle V,\nabla \mathscr{L}(\phi_i\phi_j)\rangle.  
\end{equation}
\end{lemma}
\begin{proof}
Equation~\eqref{eq.scrH.XY} implies
\begin{align*}
\mathscr{H}_T(\nabla\phi_i,\nabla\phi_j)=&-\langle\nabla_{T\nabla\phi_i}V,\nabla\phi_j\rangle
-\langle\nabla_{T\nabla\phi_j}V,\nabla\phi_i\rangle+\langle(\nabla_VT)\nabla\phi_i,\nabla\phi_j\rangle.
\end{align*}
Now, we observe that 
\begin{align*}
    \langle\nabla_{T\nabla\phi_{i}}V, \nabla\phi_{j}\rangle = \dv_{\eta}(\langle V,\nabla\phi_{j}\rangle T\nabla\phi_{i})+\lambda\langle V,\nabla\phi_{j}
\rangle\phi_{i}-\nabla^{2}\phi_{j}(V,T\nabla \phi_{i}).
\end{align*}
Moreover,
\begin{align*}
  \nabla^{2}\phi_{j}(V,T\nabla\phi_{i})+ \nabla^2\phi_i(V,T\nabla\phi_j)+\langle(\nabla_VT)\nabla\phi_i,\nabla\phi_j\rangle &=\langle\nabla (T(\nabla\phi_{i},\nabla\phi_{j})),V\rangle.
\end{align*}
Thus,
\begin{align*}
   \mathscr{H}_T(\nabla\phi_i,\nabla\phi_j)=& -\dv_{\eta}(\langle V,\nabla\phi_{j}\rangle T\nabla\phi_{i}+\langle V,\nabla\phi_{i}\rangle T\nabla\phi_{j})-\lambda\langle V,\nabla(\phi_{j}\phi_{i})\rangle+\\
   &+\langle\nabla (T(\nabla\phi_{i},\nabla\phi_{j})),V\rangle\\
   &=-\dv_{\eta}(\langle V,\nabla\phi_{j}\rangle T\nabla\phi_{i}+\langle V,\nabla\phi_{i}\rangle T\nabla\phi_{j})+\frac{1}{2}\langle V,\nabla \mathscr{L}(\phi_i\phi_j)\rangle.
\end{align*}
It finishes the proof.
\end{proof}

The next proposition provides Hadamard-type formulas for the eigenvalues of the operator $\mathscr{L}$ considering both Dirichlet and Neumann boundary conditions.  
\begin{proposition}\label{pro domain}
Let $(M,g)$ be a Riemannian manifold, $\Omega\subset M$ be a bounded domain, $f_{t}:\Omega\rightarrow (M,g)$ be an analytic family of diffeomorphisms
$(\Omega_{t}=f_{t}(\Omega))$ and $\lambda$ be an eigenvalue of $\mathscr{L}_g$ with multiplicity $m>1$. Then there exist a family of $m$ functions $\{\phi_i(t)\}
\in C^{\infty}(\Omega_{t})$ with $\langle\phi_i(t),\phi_j(t)\rangle_{L^2(\Omega_t,\dm)}=\delta_{ij}$ and a differentiable family of real numbers $\lambda_i(t)$ with
$\lambda_i(0)=\lambda$, such that they satisfy
\begin{equation}\label{eq var dom1}
\left\{
\begin{array}{ccccc}
-\mathscr{L}_{t}\phi_i(t)&=& \lambda(t)\phi_i(t) &\text{in}& \Omega\\
\mathscr{B}_\alpha(\phi_i(t))&=&0 &\hbox{on}& \partial \Omega,
\end{array}
\right.
\end{equation}
for all $t$ and $i=1,\ldots,m$, where $\mathscr{B}_{\alpha}(\phi_i)=\alpha\langle T\nabla\phi_i,\nu_t\rangle+(1-\alpha)\phi_i$ for $\alpha\in\{0,1\}$.
Moreover, we obtain the following variation formula
\begin{align}\label{eq-prop3}
(\lambda_i+\lambda_j)'\Big|_{t=0}\!\!\delta_{ij}\!=&\int_{\partial\Omega}\!\Big[\mathscr{L}(\phi_i\phi_j)\langle  V, \nu\rangle-2\langle V,\nabla\phi_{j}\rangle \langle \nabla\phi_{i},T\nu\rangle-2\langle V,\nabla\phi_{i}\rangle \langle \nabla\phi_{j},T\nu\rangle\nonumber\\
&+\langle V,\nabla\eta\rangle \langle \nabla(\phi_i\phi_j),T\nu\rangle\Big]\dn
\end{align}
where $V=\frac{d}{dt}\big|_{t=0}f_{t}$.
\end{proposition}

\begin{remark}
We point out that if $T=id$ and $\eta$ is constant, then formula~\eqref{eq-prop3} has been obtained by Soufi and Ilias~\cite{soufi}, moreover, Henry~\cite{henry} also gave a similar formula for the Laplacian on domains in $\mathbb{R}^n$.
\end{remark}

\begin{proof}
Consider the family of metrics $g(t)=f_{t}^*g$ and symmetric $(0,2)$--tensors $T_t=f_t^\ast T$ on $\Omega$. It is not difficult to see that Lemma \ref{LemExist} is still valid for the operator $\bar{\mathscr{L}}_t$ with $\eta(t)=\eta\circ f_t$, that is, there exists a family $\{\bar\phi_i(t)\}\subset C^{\infty}(\Omega)$ of analytic functions  in $t$ satisfying $\langle\bar\phi_i(t),\bar\phi_j(t)\rangle_{L^2(\Omega,\dm_t)}=
\delta_{ij}$ and
\begin{equation}\label{eq.L.var.dom}
\begin{array}{ccccc}
    -\bar{\mathscr{L}}_t\bar\phi_i(t) &=& \lambda_i(t)\bar\phi_i(t) &  \mbox{in} &\Omega.
\end{array}
\end{equation}
We claim that $\lambda_i(t)$ and $\phi_i(t):=\bar\phi_i(t)\circ f_t^{-1}$ satisfy \eqref{eq var dom1}. Initially, we have
\begin{equation*}
g(df(T_tX),df e_i)=g_t(T_tX,e_i)=T_t(X,e_i)=T(df X,df e_i)=g(T df X,df e_i),
\end{equation*}
for all $X\in\mathfrak{X}(M)$. So $dfT_t\nabla_t\bar{\phi}=T\nabla\phi$ and
\begin{equation*}
\dv_gT\nabla\phi=\dv_gdfT_t\nabla_t\bar{\phi}=\dv_{g(t)}T_t\nabla_t\bar{\phi}.
\end{equation*}
Thus, for each $q=f(p)\in\Omega_t$ we obtain
\begin{align*}
\big(\mathscr{L}_g\phi_{i}(t)\big)(q)
&= \big[\dv_g \big(T\nabla\phi_{i}(t)\big) -T\big(\nabla\eta,\nabla\phi_{i}(t)\big)\big]_{f(p)}\\
&= \big[\dv_{g(t)}\big(T_t\nabla_{t}\bar{\phi_{i}}(t)\big)   -T_{t}\big(\nabla_t\eta(t),\nabla_t\bar{\phi_{i}}(t)\big)\big]_p\\
&= \big[\bar{\mathscr{L}}_t\bar{\phi_{i}}(t)\big]_p =-\lambda_{i}(t)\big(\bar{\phi_{i}}(t)\big)_p
=-\lambda_{i}(t)\big(\phi_{i}(t)\circ f\big)_{p} =  -\lambda_i(t)\phi_{i}(t)(q).
\end{align*}

Since $\bar\phi_{i}(0)=\phi_i(0)$, $\bar{\mathscr{L}}_0=\mathscr{L}$ and $h=\langle H,g\rangle=2\dv V$, by equations \eqref{eq.LB} and \eqref{eq.LB-TgN} we have
\begin{equation*}
s_{ij}\delta_{ij} = \int_{\Omega}\mathscr{L}(\phi_i\phi_j)\dv V + 2\mathscr{H}_T(\nabla\phi_i,\nabla\phi_j)
+T(\nabla\dot{\eta},\nabla(\phi_{i}\phi_{j}))\dm,
\end{equation*}
where $s_{ij}=(\lambda_i+\lambda_j)'$ and $\dot{\eta}=\frac{d}{dt}\big|_{t=0}(\eta\circ f_t)=\langle\nabla\eta,V\rangle
$.  By Lemma \ref{lem.domH} and the identity 
\begin{align*}
\dv_\eta(\dot{\eta}T(\nabla(\phi_i\phi_j)))&=\dot{\eta}\mathscr{L}(\phi_i\phi_j)+T(\nabla\dot{\eta},\nabla(\phi_i\phi_j))\\
&=\langle \mathscr{L}(\phi_i\phi_j) V,\nabla\eta\rangle+T(\nabla\dot{\eta},\nabla(\phi_i\phi_j)),\end{align*}
we obtain
\begin{align*}
s_{ij}\delta_{ij} = & \int_{\Omega}\Big[\mathscr{L}(\phi_i\phi_j)\dv V -2\dv_{\eta}(\langle V,\nabla\phi_{j}\rangle T\nabla\phi_{i}+\langle V,\nabla\phi_{i}\rangle T\nabla\phi_{j})\\
&+\langle V,\nabla \mathscr{L}(\phi_i\phi_j)\rangle+\dv_\eta(\langle V,\nabla\eta\rangle T\nabla(\phi_i\phi_j))- \langle \mathscr{L}(\phi_i\phi_j) V,\nabla\eta\rangle\Big]\dm\\
=&\int_{\Omega}\Big[\dv_{\eta}(\mathscr{L}(\phi_i\phi_j)V-2\langle V,\nabla\phi_{j}\rangle T\nabla\phi_{i}-2\langle V,\nabla\phi_{i}\rangle T\nabla\phi_{j}\\
&+\langle V,\nabla\eta\rangle T\nabla(\phi_i\phi_j)\Big]\dm\\
=&\int_{\partial\Omega}\Big[\mathscr{L}(\phi_i\phi_j)\langle  V, \nu\rangle-2\langle V,\nabla\phi_{j}\rangle \langle T\nabla\phi_{i},\nu\rangle-2\langle V,\nabla\phi_{i}\rangle \langle T\nabla\phi_{j},\nu\rangle\\
&+\langle V,\nabla\eta\rangle \langle T\nabla(\phi_i\phi_j),\nu\rangle\Big]\dn,
\end{align*}
which is~\eqref{eq-prop3}.
\end{proof}

Proposition~\ref{pro domain} under the boundary conditions becomes 
\begin{proposition}\label{Hadamard for domains}
Let $f_t,\lambda(t),\phi_i(t)$ be as in Proposition~\ref{pro domain}, and let us consider the eigenvalue boundary problem
\begin{equation}
\left\{
\begin{array}{ccccc}
-\mathscr{L}_{t}\phi_i(t)&=& \lambda(t)\phi_i(t) &\text{in}& \Omega\\
\mathscr{B}_\alpha(\phi_i(t))&=&0 &\hbox{on}& \partial \Omega.
\end{array}
\right.
\end{equation}
Then we have
\begin{align}
for\;\;\alpha=0\;\; case:\; (\lambda_i+\lambda_j)'\delta_{ij}&=-2\int_{\partial\Omega}\frac{\partial\phi_i}{\partial\nu}\frac{\partial\phi_j}{\partial\nu}T(\nu,\nu)\langle V,\nu\rangle\dn,\label{DP}\\ 
for\;\;\alpha=1\;\; case:\; (\lambda_i+\lambda_j)'\delta_{ij}&=2\int_{\partial\Omega}(T(\nabla\phi_i,\nabla\phi_j)-\lambda\phi_i\phi_j)\langle V,\nu\rangle\dn. \label{NP}
\end{align}
\end{proposition}
\begin{proof}
It is enough to use \eqref{eq-prop3} together with the identities  
\begin{align*}
\mathscr{L}(\phi_i\phi_j)& = \phi_i\mathscr{L}\phi_j + \phi_j\mathscr{L}\phi_i + 2 T(\nabla \phi_i, \nabla\phi_j)
\end{align*}
and
\begin{align*}
    \nabla\phi_i&=\nabla_{\partial\Omega}\phi_i+\frac{\partial\phi_i}{\partial \nu}\nu\,\,\quad{on}\,\,\partial\Omega
\end{align*}
to obtain \eqref{DP} and \eqref{NP}.
\end{proof}

It is known that the set $\mathrm{Diff}^r(\Omega)$ of $\mathcal{C}^r$--diffeomorphisms of $\Omega$ is an affine submanifold of a Banach Space \cite{delfour}. The result below shows that the multiplicity of an eigenvalue can be reduced by small perturbation of the domain.

\begin{theorem}\label{thm2}
Let $(M,g)$ be a Riemannian manifold and let $\Omega\subset M$ be a bounded domain. If $\lambda$ is an eigenvalue with multiplicity $m>1$ of the problem
\begin{equation}\label{problem.domain}
\left\{\begin{array}{ccccc}
    -\mathscr{L}_g\phi&=&\lambda\phi & \mbox{in} &\Omega\\
    \mathscr{B}_\alpha \phi&=&0 & \mbox{on}& \partial \Omega,
\end{array}\right.
\end{equation}
where $\mathscr{B}_{\alpha}(\phi_i)=\alpha\langle T\nabla\phi_i,\nu_t\rangle+(1-\alpha)\phi_i$ for $\alpha\in\{0,1\}$,
then in every neighborhood of the identity $id_\Omega$ with respect to the $\mathcal{C}^r$ topology, there exist a diffeomorphism $f$ and a positive $\epsilon$ such that $|\lambda(f)-\lambda(id_\Omega)|< \epsilon$, and $\lambda(f)$ is simple. In particular, the subset of diffeomorphisms $\mathfrak{D}\subset \mathrm{Diff}^r({\Omega})$ that make the eigenvalues of \eqref{problem.domain} simple, is residual.
\end{theorem}
\begin{proof} Let $\lambda$ be an eigenvalue of multiplicity $m>1$ and suppose that, for all perturbation by diffeomorphism of $\Omega$, the multiplicity of $\lambda$ cannot be reduced. The proceeding follows the same lines as in proof of Theorem \ref{thm1}. 

For $\alpha=0$ case, from \eqref{DP} we have $\frac{\partial\phi_i}{\partial \nu}
\frac{\partial\phi_j}{\partial \nu} = 0$ on $\partial\Omega$. This way either $\frac{\partial\phi_i}{\partial \nu} = 0$ or $\frac{\partial\phi_j}
{\partial \nu} = 0 $ in some open set $U$ of $\partial \Omega$. If $\frac{\partial\phi_i}{\partial \nu} = 0$ in $U$, since $\phi_i = 0 $ on $\partial \Omega$,
it follows from the unique continuation principle \cite{homander} that $\phi_i = 0$ on $\Omega$, which is a contradiction.

For $\alpha=1$ case. From \eqref{NP} we have
\begin{align*}
    T(\nabla\phi_i,\nabla\phi_j)-\lambda\phi_i\phi_j=0
\end{align*}
on $\partial \Omega$. By the Uhlenbeck's argument once again we get a contradiction. It proves the first part of the theorem. The second part follows from the analogous argument as in the proof of Theorem~\ref{thm1}.
\end{proof}

\section{Application to extremal domains for the k--eigenvalue}\label{ApExtremal}

Before to claim the main results of this section we will begin with some definitions and remarks. Here, we will consider only analytic volume-preserving deformation of $\Omega\subset M$. Let $\mu_k(t)$ be the $k^{th}$--eigenvalue of $\mathscr{L}_t$ with Dirichlet boundary condition.

If $\Omega_t=f_t(\Omega)$ is an analytic volume-preserving deformation of $\Omega$, then it is not difficult to see that $v:=\big\langle \frac{d}{dt}\Big|_{t=0}f_t,\nu\big\rangle$ must satisfy $\int_{\partial\Omega}v\dn=0$. We denote $\mathcal{A}_0(\partial\Omega)$ the
set of all regular functions on $\partial\Omega$ such that $\int_{\partial\Omega}v\dn=0$. Soufi and Ilias proved that given a $v\in \mathcal{A}_0(\partial\Omega)$ there exists an analytic volume-preserving deformation $\Omega_t=f_t(\Omega)$ such that $v:=\big\langle \frac{d}{dt}\Big|_{t=0}f_t,\nu\big\rangle\in \mathcal{A}_0(\partial\Omega)$ (see \cite{soufi}). Recall that a domain $\Omega\subset M$ is a local minimizer (local maximizer) for the $k^{th}$--eigenvalue $\mu_k$ of $\mathscr{L}$ if for any analytic volume-preserving deformation $\Omega_t$, the function $t\mapsto \mu_k(t)$ admits a local minimum (local maximum) at $t=0$.

\begin{theorem}\label{thm3}
Let $k$ be a positive integer such that the $k^{th}$--eigenvalue  $\mu_k$ of $\mathscr{L}\phi=\dv_\eta(T\nabla\phi)$ with Dirichlet boundary condition satisfies $\mu_k>\mu_{k-1}$ (resp. $\mu_k<\mu_{k+1}$). If $\Omega$ is a local minimizer (resp. local maximizer) for $\mu_k$, then, it is simple and for some constant $c$ its associated eigenfuction $\phi$ satisfies
\begin{equation*}\left
|\frac{\partial\phi}{\partial\nu}\sqrt{T(\nu,\nu)}\right|=c\,\,\,\mbox{on}\,\,\, \partial\Omega.
\end{equation*}
\end{theorem}
\begin{proof}
Suppose $\mu_k>\mu_{k-1}$ and consider $v=\big\langle \frac{d}{dt}\Big|_{t=0}f_t,\nu\big\rangle\in\mathcal{A}_0(\partial\Omega)$ such that $\Omega_t=f_t(\Omega)$ is a volume-preserving analytic deformation of $\Omega$. Let $\{\lambda_{i}(t)\}$ and $\{\phi_{i}(t)\}$ be analytic families of eigenvalues and eigenfunctions given by Lemma~\ref{LemExist} such that  $\lambda_i(0)=\mu_k$. By continuity $\lambda_{i}(t)>\mu_{k-1}(t)$ for sufficiently small $t$, since $\lambda_{i}(0)=\mu_k>\mu_{k-1} $.
By hypothesis, the function $t\mapsto\mu_{k}(t)$ admits a local minimum at $t=0$. So, $\lambda_i(t)$ also admits a local minimum at $t=0$ and then $\left.\frac{d}{dt}\lambda_{i}(t)\right|_{t=0}=0$. Eq.~\eqref{DP} of Proposition~\ref{Hadamard for domains} give us
\begin{equation*}
\int_{\partial\Omega}v\left(\frac{\partial\phi}{\partial\nu}\right)^2T(\nu,\nu)d\mu=0
\end{equation*} 
for all $\phi\in E_k$ and $v\in\mathcal{A}_0(\partial\Omega)$, where $E_k$ is the eigenspace associated to $k^{th}$--eigenvalue $\mu_k$. Hence, $\frac{\partial\phi}{\partial\nu}\sqrt{T(\nu,\nu)} $ is locally constant on $\partial\Omega$ for any $\phi\in E_k$. Let  $\phi_1$ and $\phi_2$ be two
eigenfunctions in $E_k$, one can find a linear combination $\alpha\phi_1+\beta\phi_2=:\phi$ so that $\frac{\partial\phi}{\partial\nu}$ vanishes, at
least, on one connected component of $\Omega$. To this end, it  is sufficient to choose $\alpha$ and $\beta$ such that
\begin{equation*}
\alpha\frac{\partial\phi_1}{\partial\nu}\sqrt{T(\nu,\nu)}=-\beta\frac{\partial\phi_2}{\partial\nu}\sqrt{T(\nu,\nu)}.
\end{equation*}
Applying the principle of the unique continuation \cite{homander}, we deduce that $\phi$ is identically zero and that $\mu_k$ is simple.
To complete the proof, we must show that, for all $\phi\in E_k$,
\begin{equation*}
\left(\frac{\partial\phi}{\partial\nu}\right)^2T(\nu,\nu)
\end{equation*}
takes the same constant value on all the connected components of $\partial\Omega$. For it, let
$\Sigma_1$ and $\Sigma_2$ be two distinct connected components of $\partial\Omega$ and let $v\in\mathcal{A}_0(\partial\Omega)$ be the function given by
$v=vol(\Sigma_2)$ on $\Sigma_1$, $v=-vol(\Sigma_1)$ on $\Sigma_2$ and $v=0$ on the others components. Then the condition $\int_{\partial\Omega}v
\left(\frac{\partial\phi}{\partial\nu}\right)^2T(\nu,\nu)\dn=0$ implies that
\begin{equation*}
\left.\left(\frac{\partial\phi}{\partial\nu}\right)^2T(\nu,\nu)\right|_{\Sigma_1} =\left.\left(\frac{\partial\phi}{\partial\nu}\right)^2T(\nu,\nu)
\right|_{\Sigma_2}
\end{equation*}
which completes the proof of our assertion. Notice that the same arguments hold in the case $\mu_k<\mu_{k+1}$.
\end{proof}

\section{Applications to Ricci flow on closed manifold}\label{ApRF}

In this section, we study some properties of the eigenvalues of $\mathscr{L}_{g(t)}$ along the Ricci flow equation
\begin{align*}
    \frac{d}{dt}g(t)=-2Ric_{g(t)}
    \end{align*} 
on an $(n\geq 3)$-dimensional closed smooth manifold $M^n$, where $Ric_{g(t)}$ stands for the Ricci curvature of the Riemannian manifold $(M,g(t)).$ 

Hamilton~\cite{hamilton1} established the existence and the uniqueness of solutions to the Ricci flow equation in a maximal interval $[0,\delta),$ $\delta\leq +\infty$, for any given initial Riemannian metric $g_0=g(0).$ This maximal solution is then called the \emph{Ricci flow} with initial condition $g_0,$ and $\delta$ (whenever finite) is called the \emph{blow-up time} of the flow.

We start by observing that the eigenvalues of $\mathscr{L}_{g(t)}(\cdot)= \dv_{\eta}(T_{g(t)}\nabla(\cdot))$ can be parametrized  $\mathcal{C}^2$--differentiable in $t$, see Theorem $(C)$ in \cite{kriegl}. Now, we derive a general evolution formula for the eigenvalues of $\mathscr{L}_{g(t)}$ along the Ricci flow.

\begin{proposition}\label{prop var eigenv rf}
If $\lambda(t)$ denotes the evolution of an eigenvalue of $\mathscr{L}_{g(t)}$ along the Ricci flow on $M^n$, then
\begin{equation}\label{EQ5not normali}
\lambda'=\int_{M}R(\lambda u^{2}-T(\nabla u,\nabla u))\dm + \int_{M}\Big[4\ricci(T\nabla u,\nabla u) + T'(\nabla u,\nabla u)\Big]\dm,
\end{equation}
where $u$ stands for an eigenfunction associated to the eigenvalue $\lambda$, and $R$ is the scalar curvature.
\end{proposition}
\begin{proof}
Taking $H=-2\ricci_{g(t)}$ we get
\begin{equation*}
TH(\nabla u,\nabla u)=HT(\nabla u,\nabla u)=-2Ric(T\nabla u,\nabla u).
\end{equation*}
Thus, by definition of $\mathscr{H}_T$ we have
\begin{equation*}
\mathscr{H}_T(\nabla u,\nabla u)=4Ric(T\nabla u,\nabla u)+T'(\nabla u,\nabla u).
\end{equation*}
Since $h=-2R$ and $\mathscr{L}u^{2}= 2u\mathscr{L}u+2T(\nabla u,\nabla u)$, we obtain from \eqref{eq.LB}
\begin{eqnarray*}
\lambda'(t)&=&\int_M\frac{h}{4}\mathscr{L}u^{2}+\mathscr{H}_T(\nabla u,\nabla u)\dm\\
&=& \int_M R(\lambda u^2-T(\nabla u,\nabla u))\dm+\int_M4Ric(T\nabla u,\nabla u) +T'(\nabla u,\nabla u)\dm,
\end{eqnarray*}
which is \eqref{EQ5not normali}.
\end{proof}

Let us now consider the behavior of the spectrum when we evolve an initial metric that is homogeneous. Hamilton showed that the Ricci flow preserves the isometries of the initial Riemannian manifold. Hence, the evolving metric remains homogeneous along the flow. This important observation implies the following:

\begin{theorem}\label{rem lm Ricci}
Let $\lambda(t)$ be the evolution of an eigenvalue of $\mathscr{L}_{g(t)}$ along the Ricci flow on a closed homogeneous Riemannian manifold $M^n$. If $T'\geq -4Ric(T,\cdot)$, then $\lambda(t)$ is non-decreasing along the flow.
\end{theorem}
\begin{proof}
Since the evolving metric remains homogeneous along the flow, we have that $R$ is constant for all time. So, from equation \eqref{Formula-IP} and Proposition~\ref{prop var eigenv rf} we deduce that
\begin{equation*}
\lambda'=\int_{M} 4Ric(T\nabla u,\nabla u) + T'(\nabla u,\nabla u)\dm
\end{equation*}
and the result follows from the assumption that $T'\geq -4Ric(T,\cdot)$.
\end{proof}

\begin{example}
Let $g(t)$ be a solution to the Ricci flow on a closed homogeneous Riemannian manifold $M^n$ such that the initial Riemannian metric has strictly positive Ricci curvature and it continues so for all time. Consider $T_t=\psi g(t)$, for some positive smooth function $\psi$ on $M^n$. Then, $T'> -4Ric(T,\cdot)$ and hence the eigenvalue $\lambda(t)$ of $\dv_{\eta}(\psi\nabla(\cdot))$ is non-decreasing along the flow. 
\end{example}

For a three-dimensional closed smooth manifold, Hamilton also proved that if the initial Riemannian metric has strictly positive Ricci curvature, then it continues so for all time.

\begin{theorem}\label{thmMono}
Let $g(t)$ be a solution to the Ricci flow on a three-dimensional closed Riemannian manifold $M^3$ with initial Ricci curvature that is strictly positive. Then, there exist $t_0\in[0,\delta)$ such that the eigenvalue $\lambda(t)$ of  $\dv_{\eta}(\psi\nabla(\cdot))$ increase for $t\in[t_{0},\delta).$
\end{theorem}
\begin{proof}
Let $g(t)$ be a solution to the Ricci flow on a closed three--manifold $M$ with strictly positive initial Ricci curvature, then $R>0$ and there is $\varepsilon>0$ such that $Ric\geq\varepsilon Rg$, and we can assume $2\varepsilon-1\leq0$. So both conditions remain valid on $0\leq t<\delta$ (see \cite[Theorem~9.4]{hamilton1}), and then by \eqref{EQ5not normali}
\begin{align*}
\lambda'(t)&= \lambda\int_{M}u^{2}R\dm -\int_{M}R\psi|\nabla u|^{2}\dm +2\int_{M}\psi Ric(\nabla u,\nabla u)\dm\\
           &\geq\lambda R_{\min}(t) -\int_{M}R\psi|\nabla u|^{2}\dm+2\varepsilon\int_{M}R\psi|\nabla u|^{2}\dm\\
           &\geq\lambda(R_{\min}(t)+(2\varepsilon-1)R_{\max}(t)).
\end{align*}
By Theorem~15.1 in~\cite{hamilton1}, we have $\frac{R_{\max}(t)}{R_{\min}(t)}\rightarrow 1$ as $t\rightarrow \delta$. Then there is $t_0\geq0$ so that
\[-(2\varepsilon -1) < \frac{R_{\min}(t)}{R_{\max}(t)}\leq 1 \qquad \forall t_0\leq t< \delta,\]
from which we get $\lambda'(t)> 0$ for all $t\in[t_0,\delta)$, which proves our theorem.
\end{proof}

Let $\drift u=\dv_{\eta}(\nabla u)$ be the drifted Laplacian and $Ric_\eta:=Ric+\nabla^2 \eta$ be the Bakry-\'Emery-Ricci tensor. From the classical Bochner's formula, we obtain
\begin{equation}\label{eq bochner drifiting}
\frac{1}{2}\drift|\nabla u|^2=Ric_\eta(\nabla u,\nabla u)+\langle \nabla(\drift u),\nabla u\rangle+|\nabla^2 u|^2.
\end{equation}
Besides, it is true that
\begin{align*}
\langle \nabla\psi,\nabla |\nabla u|^2 \rangle = 2\langle \nabla_{\nabla u} \nabla u, \nabla\psi \rangle
= 2(\nabla u)\langle \nabla \psi, \nabla u \rangle -2\nabla^2\psi(\nabla u, \nabla u),
\end{align*}
for any smooth function $\psi$ on $M$. So,
\begin{align}
 \psi \drift|\nabla u|^2 &= \mathscr{L}|\nabla u|^2-\langle \nabla\psi,\nabla |\nabla u|^2 \rangle \nonumber\\
         &=  \mathscr{L}|\nabla u|^2 -2(\nabla u)\langle \nabla \psi, \nabla u \rangle +2\nabla^2\psi(\nabla u, \nabla u) \label{eq boc d1}.
\end{align}
But,
\begin{align}
(\nabla u)\langle \nabla \psi, \nabla u \rangle
       &= \langle \nabla u,\nabla \langle\nabla u,\nabla\psi \rangle \rangle \nonumber\\
       &= \langle \nabla u,\nabla(\mathscr{L}u)\rangle -\langle \nabla u, \nabla(\psi \drift u)\rangle\nonumber\\
       &= \dv_{\eta}(\mathscr{L}u\nabla u) -(\mathscr{L}u)\drift u -\dv_{\eta}\big(\psi\drift u\nabla u\big) +\psi(\drift u)^2 \label{eq boc d2}
\end{align}
and
\begin{equation}\label{eq boc d3}
\psi\langle \nabla(\drift u),\nabla u \rangle = \dv_{\eta}\big(\psi\drift u\nabla u\big) -(\mathscr{L}u)\drift u.
\end{equation}
Joining \eqref{eq bochner drifiting}, \eqref{eq boc d1}, \eqref{eq boc d2} and \eqref{eq boc d2}, we get the Bochner type formula for the operator
$\mathscr{L}u=\dv_\eta(\psi\nabla u)$ as follows
\begin{eqnarray}\label{eq boc psi g}
\frac{1}{2}\mathscr{L}|\nabla u|^2 &=&\dv_{\eta}(\mathscr{L}u\nabla u) +\psi \ricci_\eta(\nabla u,\nabla u) -2(\mathscr{L}u)\drift u +\psi(\drift u)^2
\nonumber\\ 
&&+\psi|\nabla^2 u|^2 -\nabla^2\psi(\nabla u, \nabla u).
\end{eqnarray}
This formula will be useful in the proof of our next result.

\begin{remark}
Alencar, Neto and Zhou~\cite{AGD} showed a Bochner type formula for the operator that has been introduced by Cheng and Yau~\cite{Cheng-Yau}. 
The Bochner type formula for the more general expression of $\mathscr{L}$ has been proved by Gomes and Miranda~\cite{miranda}.    
\end{remark}

We know that in the hypothesis of Theorem~\ref{thmMono} the solution to the Ricci flow becomes extinct in finite time, in particular, $\lim\limits_{t\rightarrow \delta}R_{\min}(t)=\infty$.
\begin{theorem}\label{AAAE}
Let $g(t)$ be the solution to the Ricci flow on a three-dimensional closed Riemannian manifold $M^3$ with strictly positive initial Ricci curvature and $\lambda(t)$ the evolution of an eigenvalue of $\dv_\eta(\psi\nabla(\cdot))$, with $\psi\geq c$ for some constant $c>0$. If $\nabla^2\psi\leq0$, then
\begin{equation*}
\lim\limits_{t\rightarrow \delta}\lambda(t)=\infty.
\end{equation*}
\end{theorem}
\begin{proof}
Integrating \eqref{eq boc psi g} gives
\begin{eqnarray*}
\int_M\psi Ric(\nabla u,\nabla u)\dm
&=& 2\int_M(\mathscr{L}u)\drift u\dm -\int_M\psi(\drift u)^2\dm -\int_M\psi|\nabla^2u|^2\dm\nonumber\\
&&  +\int_M\nabla^2\psi(\nabla u,\nabla u)\dm -\int_M\psi\nabla^2\eta(\nabla u,\nabla u)\dm\\
&\leq& -2\lambda\int_Mu(\drift u)\dm -\int_M\psi\nabla^2\eta(\nabla u,\nabla u)\dm \\
&=& 2\lambda\int_M|\nabla u|^2\dm-\int_M\psi\nabla^2\eta(\nabla u,\nabla u)\dm\\
&\leq& 2\lambda c^{-1}\int_M\psi|\nabla u|^2\dm-\int_M\psi\nabla^2\eta(\nabla u,\nabla u)\dm\\
&=& 2\lambda^2c^{-1}-\int_M\psi\nabla^2\eta(\nabla u,\nabla u)\dm.
\end{eqnarray*}

We already know that for any solution of the Ricci flow on a closed three-manifold with positive Ricci curvature, there exists $\epsilon>0$ such that
$\ricci\geq\epsilon Rg$ is preserved along the flow. Thus
\begin{equation*}
2\lambda^2c^{-1}-\int_M\psi\nabla^2\eta(\nabla u,\nabla u)\dm \geq \epsilon \int_MR\psi|\nabla u|^2\dm \geq \epsilon R_{\min}\lambda.
\end{equation*}
Now, note that $\nabla^2\eta(\nabla u,\nabla u)\geq d |\nabla u|^2$, for some constant $d(t)>0$, and then
\begin{equation*}
\lambda(t)\geq c(\epsilon R_{\min}(t)+d(t))/2.
\end{equation*}
Since
\begin{equation*}
\lim\limits_{t\rightarrow \delta}R_{\min}(t)=\infty,
\end{equation*}
the proof is complete.
\end{proof}

\textbf{Acknowledgements:} The first author gratefully acknowledges financial support from the Fundação de Amparo à Pesquisa do Estado do Amazonas (FAPEAM) through a doctoral scholarship. The second author has been partially supported by the Conselho Nacional de Desenvolvimento Científico e Tecnológico (CNPq), and by the Fundação de Amparo à Pesquisa do Estado de São Paulo (FAPESP), Grants 2022/16097-2, 2023/11126-7 and 2024/00923-6. The third author was also partially supported by FAPESP, Grants 2016/10009-3 and 2020/14075-6. The authors thank D. Tsonev for valuable comments, discussions and encouragement, and the referee for the careful reading and for the useful comments that improved the paper.


\begin{thebibliography}{99}
\bibitem{AGD} H. Alencar, G.S. Neto and D. Zhou, Eigenvalue estimates for a class of elliptic differential operators on compact manifolds, Bull. Braz. Math. Soc. (N.S.) 46 (3) (2015), 491--514.
\bibitem{MCAFandGomes} M.C. Araújo Filho and J.N.V. Gomes, Estimates of eigenvalues of an elliptic differential system in divergence form, Z. Angew. Math. Phys. 73 (2022) 210.
\bibitem{bando} S. Bando and H. Urakawa, Generic properties of the eigenvalue of the Laplacian for compact Riemannian manifolds, Tohoku Math. J. 35 (1983) 155--172.
\bibitem{berger} M. Berger, Sur les premières valeurs propres des variétés Riemanniennes, Compos. Math. 26 (1973) 129--149.
\bibitem{canzani} Y. Canzani, On the multiplicity of eigenvalues of conformally covariant operators, Ann. Inst. Fourier 64 (3) (2014) 947--970.
\bibitem{Cao_annalen}X. Cao, S. Hou and J. Ling, Estimate and monotonicity of the first eigenvalue under the Ricci flow, Math. Ann. 354 (2) (2012) 451--463.
\bibitem{Cheng-Yau} S.Y. Cheng and S.T. Yau, Hypersurfaces with constant scalar curvature, Math. Ann. 225 (1977), 195--204.
\bibitem{delfour} M.C. Delfour and J.-P. Zolesio, Shapes and Geometries: Analysis, Differential Calculus, and Optimization. Society for Industrial and Applied Mathematics, 2001.
\bibitem{miranda} J.N.V. Gomes and J.F.R. Miranda, Eigenvalue estimates for a class of elliptic differential operators in divergence form, Nonlinear Anal. 176 (2018) 1--19.
\bibitem{gomes} J.N.V. Gomes and M.A.M. Marrocos, On eigenvalue generic properties of the Laplace-Neumann operator, J. Geom. Phys. 135 (2019) 21--31.
\bibitem{hamilton1} R.S. Hamilton, Three-manifolds with positive Ricci curvature, J. Differential Geom. 17 (1982) 255--306.
\bibitem{henry} D.B. Henry, Perturbation of the boundary in boundary-value problems of partial differential equations, Cambridge University Press, 2005.
\bibitem{homander} L. Hörmander, Linear partial differential operators, Springer, New York, 1963.
\bibitem{kato} T. Kato, Perturbation Theory for Linear Operators, Springer, 1980.
\bibitem{kriegl} A. Kriegl and P.W. Michor, Differentiable perturbation of unbounded operators, Math. Ann. 327 (2003) 191--201.
\bibitem{lions} J.L. Lions and E. Magenes, Non-homogeneous boundary value problems and applications, Die Grundlehren der mathematischen Wissenschaften 181 (1972) Springer, New York.
\bibitem{marrocos} M.A.M. Marrocos and A.L. Pereira, Eigenvalues of the Neumann Laplacian in symmetric regions, J. Math. Phys. 56 (2015) 111502.
\bibitem{Navarro} J. Navarro, On second-order, divergence-free tensors, J. Math. Phys. 55 (2014) 062501. 
\bibitem{Serre} D. Serre, Divergence-free positive symmetric tensors and ﬂuid dynamic, Ann. Inst. H. Poincaré Anal. Non Linéaire 35 (5) (2018), 1209--1234.
\bibitem{soufi} A. El Soulfi and S. Ilias, Domain deformations and eigenvalues of the Dirichlet Laplacian in a Riemannian manifolds, Illinois J. Math. 51 (2007) 645--666.
\bibitem{uhlenbeck} K. Uhlenbeck, Generic Properties of Eigenfunctions, Amer. J. Math. 98 (4) (1976) 1059--1078.
\bibitem{wall} C.T.C. Wall, Singular Points of Plane Curves, London Mathematical Society Student Texts, 2004.
\end{thebibliography}
\end{document}